\documentclass[12 pt]{amsart}
\usepackage{graphicx}
\usepackage[colorlinks= true, urlcolor=black, citecolor=blue]{hyperref}

\newtheorem{theorem}{Theorem}[section]
\newtheorem{lemma}[theorem]{Lemma}
\newtheorem{corollary}{Corollary}[theorem]
\newtheorem{proposition}{Proposition}[section]

\newtheorem{definition}[theorem]{Definition}
\newtheorem{example}[theorem]{Example}

\newtheorem{remark}[theorem]{Remark}

\numberwithin{equation}{section}

\newcommand{\abs}[1]{\lvert#1\rvert}

\begin{document}
\title{Contracting Boundary of a Cusped Space}

\author{ Abhijit Pal}
\address{Indian Institute of technology, Kanpur}

\author{Rahul Pandey}
\address{Indian Institute of technology, Kanpur}

%    General info
\subjclass[2000]{}

\date{}

\dedicatory{}

\keywords{Contracting Boundary, Relatively Hyperbolic Groups}

\begin{abstract}
Let $G$ be a finitely generated group. Cashen and Mackay proved that if the contracting boundary of $G$ with the topology of fellow travelling quasi-geodesics is compact then $G$ is a hyperbolic group. Let $\mathcal{H}$ be a finite collection of finitely generated infinite index subgroups of $G$. Let $G^h$ be the cusped space obtained by attaching combinatorial horoballs to each left cosets of elements of $\mathcal {H}$. In this article, we prove that if the combinatorial horoballs are contracting and $G^h$ has compact contracting boundary then $G$ is hyperbolic relative to $\mathcal{H}$.
\end{abstract}
\maketitle

\begin{center}{AMS Subject Classification : 20F65, 20F67, 57M07}\end{center}
\section{Introduction}
The Gromov boundary of a proper hyperbolic metric space provides a compactification of the space by associating  a set of points at infinity.
To construct an analogue of the  Gromov boundary for a general proper geodesic metric space has created a lot of interest in recent times. In \cite{Charney15}, Charney-Sultan gave the notion of ‘contracting' boundary for a CAT$(0)$ space which is an analogue of the Gromov boundary. A subset $Z$ of a geodesic metric space $X$ is said to be $\rho$-contracting for a sublinear function $\rho$ if the diameter of projection of balls of radius $r$ disjoint with $Z$ is at most $\rho(r)$ (see Definition \ref{ContractingSub}). Contracting boundary of a geodesic metric space is defined to be the collection of all contracting geodesic rays starting from a fixed basepoint modulo Hausdorff equivalence.  Contracting boundary has been studied under two different topologies,  direct limit topology by Cordes(\cite{Morse2015}) and topology of fellow travelling quasi-geodesics by Cashen and Mackay(\cite{cashen2017}). Both the topologies are quasi-isometric invariant. In general, topology defined by Cashen and Mackay is coarser than the direct limit topology and for finitely generated groups it is metrizable. In this article, we will deal with contracting boundary having topology of fellow travelling
quasi-geodesics. The contracting boundary of a geodesic metric space $X$ will be denoted as $\partial_cX$.\par
Relatively hyperbolic groups are generalisations of  hyperbolic groups and it has been defined in many ways  by various people (See \cite{Farb}, \cite{BowditchRel}, \cite{Drutu}, \cite{DrutuSapir}, \cite{GrovesManning}).
We refer the reader to see the article \cite{hruska} by Hruska for various equivalent definitions of relatively hyperbolic groups.
In this article,  we take the definition of relatively hyperbolic groups given by Groves and Manning (\cite{GrovesManning}). A finitely generated group $G$ is said to be hyperbolic relative to a finite collection $\mathcal{H}=\{H_i\}$ of finitely generated  subgroups if the cusped space $G^h$ obtained from a Cayley graph of $G$ by attaching combinatorial horoballs to each left coset of $H_i$ is a hyperbolic metric space (See Definition \ref{cuspedspace}). If $G$ is hyperbolic relative to $\mathcal{H}$ then the cusped space, being hyperbolic, has compact contracting boundary and the combinatorial horoballs are quasi-convex in the cusped space with vertical rays in each combinatorial horoballs are contracting in the cusped space. In \cite{cashen2017}, Cashen and Mackay proved that if a  finitely generated group has compact contracting boundary then the group is hyperbolic. In this article, we generalize Cashen and Mackay's result to relatively hyperbolic groups.
Our main theorem characterizes a relatively hyperbolic group as follows :
\begin{theorem}(Theorem \ref{Mthm2})
Let $G$ be a finitely generated group and  $\mathcal{H}=\{H_i\}$ be a finite collection of finitely generated infinite index subgroups of $G$. Then, $G$ is hyperbolic relative to $\mathcal H$ if and only if
every combinatorial horoball $H_i^h$ is contracting in the cusped space $G^h$ and $G^h$ has compact contracting boundary.
\end{theorem}
Suppose $\mathcal H$ is the singleton set consisting of the trivial subgroup of $G$ in Theorem \ref{Mthm2}. Then combinatorial horoballs for each $g\in G$ corresponds to a geodesic ray $g^h:[0,\infty)\to G^h$ such that $g^h([0,\infty))\cap G=\{g\}$. It is easy see that if $\{g_n^h(\infty)\}$ is a sequence of distinct points in $\partial_cG^h$  such that $g_n\to \zeta$
in $G\cup\partial_cG$ then $g_n^h(\infty)\to\zeta$ in $\partial_cG^h$. The contracting boundaries $\partial_cG$
and $\partial_cG^h$ are metrizable (See Corollary 8.6 of \cite{cashen2017} and Lemma \ref{metrizable}). The notions of compactness and sequentially compactness are equivalent in a metric space. If $\partial_cG$ non-empty compact, then given a sequence $\{g_n^h(\infty)\}$, we get a subsequence $\{g_{n_k}\}$ of $\{g_n\}$ such that $g_{n_k}$ converges to a point of $\partial_cG$ (see Case 1 of Proposition \ref{compactf} for a proof). So, by assuming $G$ to be a finitely generated group with $\partial_cG$ non-empty compact and $\mathcal H$ to be the singleton set consisting of the trivial subgroup, we get the contracting boundary of the cusped space $G^h$ to be compact  and the combinatorial horoballs are contracting. As an application of Theorem \ref{Mthm2}, we get $G$ relatively hyperbolic to $\mathcal H$ and as
$\mathcal H$ is trivial, the group $G$ is hyperbolic.

Note that for a general space with compact contracting boundary it  does not necessarily imply that the space is hyperbolic. For example, the wedge of a plane and a line has two points in the boundary and hence compact, but the space is not a hyperbolic metric space.
 Let $X$ be a proper geodesic metric space and $\mathcal{H}$ be
a collection of uniformly separated closed subsets of $X$. We can adapt the definition of Groves and Manning \cite{GrovesManning}, to define $X$ hyperbolic relative to $\mathcal H$. The space $X$
is said to be \textit{hyperbolic relative to $\mathcal{H}$}, if the quotient (cusped) space
$X^h$ obtained by attaching hyperbolic cones (analog of combinatorial horoballs) along each $H\in\mathcal{H}$  is a hyperbolic metric space (See Definitions 1.9, 1.10 of \cite{Mjpal}). The contracting boundary of the cusped space $X^h$ compact does not necessarily imply that $X$ is relatively hyperbolic to $\mathcal{H}$, see Example \ref{example}. The group theory plays a crucial role in proving Theorem \ref{Mthm2}.

\begin{example}\label{example} Consider the subspace $$Y=\{(x,0,0):x\in\mathbb R_{\geq 0}\}\cup \{(2m,s,t):m\in\mathbb N, s\in [-m,m],t\in\mathbb R\}$$ of $\mathbb R^3$.
Let $\alpha(t)=(t,0,0)$, where $t\in [0,\infty)$. Let $\beta_{n}^{-}:[0,\infty)\to Y$ be defined by $\beta_{n}^{-}(t)=(t,0,0)$ if $t\leq 2n$ and
$\beta_{n}^{-}(t)=(2n,0,2n-t)$ if $t\geq 2n$. Let $\beta_{n}^{+}:[0,\infty)\to Y$ be defined by $\beta_{n}^{+}(t)=(t,0,0)$ if $t\leq 2n$ and
$\beta_{n}^{+}(t)=(2n,0,t-2n)$ if $t\geq 2n$.
The contracting boundary $\partial_cY$ of $Y$ is given by the set $\{\alpha(\infty),\beta_{n}^{-}(\infty),\beta_{n}^{+}(\infty): n\in \mathbb N\}$, where $\alpha(\infty),\beta_{n}^{\pm}(\infty)$ respectively represent the equivalence classes of geodesic rays $\alpha,\beta_{n}^{\pm}$ in $\partial_cY$. With respect to the topology of fellow travelling quasi-geodesics, the sequences
$\{\beta_{n}^{-}(\infty)\}$, $\{\beta_{n}^{+}(\infty)\}$ converge to $\alpha(\infty)$ in $\partial_cY$. So, $\partial_cY$ is compact but $Y$ is not a hyperbolic metric space.\\
Take a wedge of a copy of the Euclidean plane with $Y$ along the point $(2m+1,0,0)$ for each $m\in \mathbb Z_{\geq 0}$ and index these wedged planes respectively by $E_m$ . Let us denote the resulting space by $X$. Attach a horoball to each $E_m$ and denote it by $E_m^{h}$. Denote the resulting space, after attaching horoballs, by $X^h$. $E_m^{h}$ is a Morse subset of $X^h$, also $\partial_cX^h$ is compact. But $X^h$ is not hyperbolic due to the presence of fat quasi-geodesic bigons. Hence $X$ is $not$ relatively hyperbolic with respect to uniformly $2$-separated collection of closed subspaces $\{E_m\}$.

 \end{example}

There are examples of a cusped space in which attached horoballs are contracting but the boundary of the cusped space is not compact.
\begin{example}  Let $H$ be any finitely generated group.
Let $G=H* \mathbb{Z}^2=H*\langle a,b|\ ab=ba\rangle$. Let $\mathcal{H}$$=$$\{H\}$ and $G^h$ be the cusped space for the pair $(G,\mathcal{H})$. Clearly, the combinatorial horoball $gH^h$ is contracting in $G^h$ for any coset $gH$.  $G^h$ is not a hyperbolic metric space  as it contains $\mathbb Z^2$. We can explicitly construct a sequence in $\partial_cG^h$ which does not converge to any point of $\partial_cG^h$.\par
 Let $\nu$ be the vertical ray in $H^h$ starting from $e$ (identity of $G$). The geodesic ray $\nu$ is contracting and let it represents a point $\zeta\in \partial_cG^h$.  For each $n\geq1$, define $\delta_n:=[1,a^n]*a^n\nu$, where $[1,a^n]$
 is the geodesic in $\mathbb Z^2$ joining $1,a^n$ and “$*$” means concatenation of two paths.  $\delta_n$ is a contracting geodesic ray and $\delta_n\in a^n\zeta$.  $\{a^n\zeta\}$ is a sequence of distinct points in $\partial_cG^h$ such that no subsequence of $\{a^n\zeta\}$ converges to a point in $\partial_cG^h$.
\end{example}
\begin{example}
Let $H$ be any finitely generated group.
Let $G=H\times \mathbb{Z}$. Let $\mathcal{H}$$=$$\{H\}$ and $G^h$ be the cusped space for the pair $(G,\mathcal{H})$. Any two left cosets $gH,g'H$ lie in a bounded neighborhood of each other however there exists geodesics in combinatorial horoballs $gH^h$ and $g'H^h$ whose end points lie in uniform bounded distance from each other but the geodesics
does not lie in a uniform bounded neighborhood of each other. Thus,
for any $g\in G$, the combinatorial horoball $gH^h$ is not contracting in $G^h$. The cusped space $G^h$ is not hyperbolic and  $\partial_cG^h$ is empty.
\end{example}
Let $G$ be a finitely generated group and $H\leq G$ be a normal hyperbolic subgroup such that Cannon-Thurston map for the pair $(G,H)$ exists. If $G$ contains a contracting element and $\Lambda_cH$ has at least two elements then $G$ is a hyperbolic group (See Theorem \ref{CT}).
In the last section, we give an application of Theorem \ref{Mthm2} related to Cannon-Thurston maps for relatively hyperbolic groups.

\section{Preliminaries}
\subsection{Morse \& Contracting Quasi-geodesics}
\begin{definition}
\begin{enumerate}
\item(Quasi-isometry): Let $K\geq 1,\epsilon\geq 0$. Let $(X,d_{X})$ and $(Y,d_{Y})$ be two metric spaces.
 A map $f: X \rightarrow Y$ is said to be a $(K,\epsilon)$-$quasi$-$isometric$ embedding if
$$\frac{1}{K}d_{X}(a,b)-\epsilon\leq d_{Y}(f(a),f(b))\leq Kd_{X}(a,b)+\epsilon$$ for all $a,b\in X.$
In addition, if for each $y\in Y$ there exists $x\in X$ such that $d_Y(y,f(x))\leq K$ then $f$
is said to be a $(K,\epsilon)$-$quasi$-$isometry$ between $X$ and $Y$.
\item (Quasi-geodesic): Let $X$ be metric space. A map $c: I \rightarrow X$,
where $I$ is any interval in $\mathbb{R}$ with usual metric, is
said to be $(K,\epsilon)$-$quasi$-$geodesic$ if c is a $(K,\epsilon)$-quasi-isometric embedding.
\end{enumerate}
\end{definition}

\begin{definition}
(Morse quasi-geodesic): A quasi-geodesic $\gamma$ in a geodesic metric space X is
called $N$-$Morse$ if
there exists a function $N:\mathbb{R}_{\geq 1} \times \mathbb{R}_{\geq 0} \rightarrow
\mathbb{R}_{\geq 0}$ such that if $q$ is
any $(K,\epsilon)$-quasi-geodesic with endpoints on $\gamma$ then
$q$ lies in the closed $N(K,\epsilon)$-neighborhood of $\gamma$.
We call $N$ to be the $Morse\ Gauge$ of $\gamma$.
\end{definition}

\noindent {A function $\rho$ is called \textit{sublinear} if it is non-decreasing, eventually
non-negative, and $lim_{r \rightarrow \infty} \rho(r)/r = 0$.}
\begin{definition}\label{ContractingSub}
(Contracting quasi-geodesic):
Let $\gamma:I\rightarrow X$ be a quasi-geodesic in a geodesic metric space $X$.
Let $\pi_{\gamma} : X \rightarrow \mathbb P({\gamma(I)})$ be defined as
$\pi_{\gamma}(x)=\{z \in \gamma \ |d(x, z) = d(x, \gamma(I))\}$, where $\mathbb P({\gamma(I)})$ denotes the power set of $\gamma(I)$.
The map $\pi_{\gamma}$ is called to be the closest point
projection to $\gamma(I)$.  For a sublinear function $\rho$, we say that $\gamma$
is $\rho$-$contracting$ if for all $x$ and $y$ in X:
 $$ d(x,y)\leq d(x,\gamma(I))\implies
diam(\pi_{\gamma(I)}(x)\cup\pi_{\gamma(I)}(y))\leq \rho(d(x,\gamma(I))) .$$
 We say that a quasi-geodesic $\gamma$ is contracting if it is $\rho$-$contracting$ for
some sublinear function $\rho$.
\end{definition}
\textbf{Note:} In the above two definitions one can take any subset $Z$ of $X$
instead of quasi-geodesics and then we have the definitions of Morse and contracting subsets.

\begin{lemma}\label{almosttriangle}(Lemma 3.8 of \cite{cashen2017}) Given a sublinear function $\rho$, there is a sublinear
function $\rho'\asymp\rho$ such that if $\alpha$, $\beta$, and $\gamma$ are  geodesic  triangle with $\alpha$ and $\beta$ are $\rho$-contracting then $\gamma$ is $\rho'$-contracting.
\end{lemma}

\begin{theorem}  (Theorem 1.4 of \cite{arzhan}) \label{Cont} Let $Z$ be a subset of a geodesic
metric space $X$. The followings are equivalent:
\begin{enumerate}
\item  $Z$ is Morse.
\item $Z$ is contracting.
\end{enumerate}
Moreover, the equivalence is ‘effective’, in the sense that the defining function of
one property determines the defining functions of the other.
\end{theorem}

Examples of contracting(or Morse) quasi geodesics include
quasi-geodesics in hyperbolic spaces, axis of pseudo-Anosov mapping classes in
Teichmuller space(\cite{Minsky}) etc.

\textit{Notation:} If $f$ and $g$ are functions then we say $f \preceq g$ if there
exists a constant $C > 0$ such
that $f(x) \leq Cg(Cx + c) + C$ for all $x$ . If $f \preceq g$ and $f \succeq g$
then we write $f\asymp g$.

\begin{lemma}(Lemma 6.3 of \cite{arzhan})\label{lemma1}
Given a sublinear function $\rho$ and a constant $C \geq 0$ there exists a
sublinear function $\rho' \asymp \rho$ such that if $Z \subseteq X$ and $Z' \subseteq X$ have
Hausdorff distance at most $C$ and $Z$ is $\rho$–contracting then $Z'$ is $\rho'$–contracting.
\end{lemma}

\begin{lemma}(Lemma 3.6 of \cite{cashen2017})\label{Clemma}
Given a sublinear function $\rho$ there is a sublinear function $\rho' \asymp \rho$
such that every subsegment of a $\rho$-contracting geodesic is $\rho'$-contracting.
\end{lemma}

\begin{theorem}(Geodesic Image Theorem, \cite{arzhan}) \label{GIT}
For $Z \subseteq X$, Z is $\rho'$-contracting if and only if there exists a sublinear function $\rho$ and a constant $\kappa_{\rho}$ such that for every geodesic segment $\gamma$, with end-points denoted by $x$ and $y$, if $d(\gamma,Z) \geq \kappa_{\rho}$ then diam ($\pi_{Z}(\gamma)) \leq \rho(\mbox{max}\{d(x,Z),d(y,Z)\})$.
Moreover $\rho$ and $\kappa_{\rho}$ depend only on $Z$ and vice-versa, with $\rho \asymp \rho'$.
\end{theorem}

\begin{theorem}(Quasi-geodesic Image Theorem, Theorem 4.2 \cite{cashen2017}) \label{QGIT} Let $Z \subseteq X$ be $\rho$-contracting.
Let $\beta :$$[0, T] \rightarrow X$ be a continuous $(L,A)$-quasi-geodesic segment. If $d(\beta,Z)\geq
\kappa(\rho, L,A)$ then:
$$ \mbox{diam}(\pi(\beta_0)\cup\pi(\beta_T))\leq \frac{L^2+1}{L}(A+d(\beta_T,Z))+\frac{L^2-1}{L}d(\beta_0,Z)+2\rho(d(\beta_0,Z))$$
\end{theorem}
\subsection{Bordification of group}
\begin{definition}\label{Morse Constant}
Given a sublinear function $\rho$ and constants $L \geq 1$ and $A\geq0$, define:
       $$ k(\rho,L,A):=\mbox{max}\{3A,3L^{2},1+inf\{R>0|\mbox{ for all } r\geq
R,3L^{2}\rho(r)\leq r\} \}$$
       Define:
       $$ k'(\rho,L,A):= (L^{2}+2)(2k(\rho,L,A)+A).$$
\end{definition}

\begin{definition}(Contracting Boundary) Let $X$ be a proper geodesic metric space
with basepoint $o$. Contracting Boundary of $X$, denoted by $\partial_{c}X$, defined to be the set of contracting quasi-geodesic rays based at $o$
modulo Hausdorff equivalence.
\end{definition}

\begin{proposition}(Lemma 5.2 of \cite{cashen2017}) For each $\zeta \in \partial_{c}X$:
     \begin{enumerate}
     \item The set of contracting geodesic rays in $\zeta$ is non-empty.
     \item There is a sublinear function:
     $$ \rho_{\zeta}(r):=\underset{\alpha,x,y}{sup}\ \mbox{diam}\big( \pi_{\alpha}(x)\cup
\pi_{\alpha}(y) \big)$$
     Here the supremum is taken over geodesics $\alpha \in \zeta$ and points x and y
such
     that $d(x, y) \leq d(x, \alpha) \leq r$
     \item Every geodesic in $\zeta$ is $\rho_{\zeta}$-$contracting$.
     \end{enumerate}
\end{proposition}

Given two distinct point $\eta$ and $\zeta$ in $\partial_{c}X$, we can find a contracting bi-infinite geodesic, $\gamma:(-\infty,\infty)\rightarrow X$ such that $\gamma((-\infty,0])$ is asymptotic to $\eta$ and $\gamma([0,\infty))$ is asymptotic to $\zeta$. This $\gamma$ is said to be a geodesic joining $\eta$ and $\zeta$.

 \begin{definition}
 { (\textit{Cashen-Mackay} Topology, \cite{cashen2017}):}\\
 $Notation:$ $N^{c}_{r}o=\{x\in X|\ d(o,x)\geq r\}$\\
    Let $X$ be a proper geodesic metric space. Take  $\zeta \in \partial_{c}X$. Fix
a geodesic ray $\alpha^{\zeta} \in \zeta$.
    For each $r \geq 1$, define $U(\zeta, r)$ to be the set of points  $\eta \in
\partial_{c}X$ such that for all
    $L \geq 1$ and $A \geq 0$ and every continuous $(L,A)$-quasi-geodesic ray
$\beta \in \eta$ we have
    \begin{equation}
    \tag{*}
         d(\beta, \alpha^{\zeta} \cap N^{c}_{r}o) \leq k(\rho_{\zeta},L,A).
    \label{eqn:key1}
         \end{equation}
     Define a topology on $\partial_{c}X$ by
 \begin{equation}
\tag{**}
\mathcal{F}\mathcal{Q}:=\{U \subset \partial_{c}X\ |\ \text{for all}\ \zeta \in U,
\text{there exists}\ r \geq 1, U(\zeta, r)\subset U \}
\label{eqn:key2}
\end{equation}
     The contracting boundary equipped with this topology was called to be
\textit{topology of fellow travelling quasi-geodesics} by Cashen-Mackay
\cite{cashen2017} and was denoted by $\partial^{\mathcal{F}\mathcal{Q}}_{c}X$.
\end{definition}
\begin{definition}
A bordification of a Hausdorff topological space $X$ is a Hausdorff space $Y$ containing $X$ as an open, dense subset.
\end{definition}
\begin{definition}\label{bordif}Let $X$ be a proper geodesic metric space. Let $\bar{X}=X\cup \partial_{c}X$. We are going to define a topology on $\bar{X}$. For $x\in X$ take a neighborhood basis for $x$ to be metric balls about $x$. For $\zeta \in \partial_{c}X$ take a neighborhood basis
for $\zeta$ to be sets $\bar{U}(\zeta,r)$ consisting of $U(\zeta, r)$ and points $x\in X$ such that we have $d\big(\gamma,N^c_{r}o\cap \alpha^{\zeta}\big)
\leq \kappa(\rho_{\zeta},L,A)$ for every $L\geq1$, $A\geq0$ and continuous $(L,A)$ quasi-geodesic segment with endpoints $o$ and $x$.
\end{definition}
$\bar{X}$ under the above topology defines a bordification of $X$ (Proposition 5.15, \cite{cashen2017}).\\ \\
\noindent{\textbf{Notation:}} Henceforth, by $\partial_cX$ we mean $\partial^{\mathcal{F}\mathcal{Q}}_{c}X$.
\begin{definition}
(Limit set) \label{limit set} (i) Let $A\subset X$. The limit set of $A$ in $\partial_cX$ is defined to be $\Lambda_c(A):= \overline{A}\setminus A$, where $\overline{A}$ is the closure of $A$ in $X\cup\partial_cX$.\\
(ii) If $G$ is a finitely generated group acting
properly discontinuously
on a proper geodesic metric space $X$ with basepoint $o$ we define the $limit\ set\
\Lambda_c(G) :=\overline{Go} \backslash Go$,
the topological frontier of orbit of $o$ under the $G$-action in
$\bar{X}$.\end{definition}
\subsection{Cusped spaces and relatively hyperbolic groups}
\begin{definition}(Combinatorial horoball) (Definition 3.1 of \cite{GrovesManning}) Let $\Gamma$ be a graph. A combinatorial horoball based on $\Gamma$ is graph with vertices $\Gamma^{(0)}\times (\mathbb{N}\cup \{0\})$ and following three types of edges;
\begin{enumerate}
\item If e is an edge of $\Gamma$ joining $v$ to $w$ then there is a corresponding edge
$\bar{e}$ connecting $(v, 0)$ to $(w, 0)$.
\item If $k > 0$ and $0 < d_{\Gamma}(v,w)\leq 2^k$, then there is a single edge connecting
$(v, k)$ to $(w, k)$.
\item If $k \geq0$ and $v \in\Gamma^{(0)}$, there is an edge joining $(v, k)$ to $(v, k + 1)$.
  \end{enumerate}
 \end{definition}
 For  proper graphs, a combinatorial horoball is a proper hyperbolic geodesic metric space with a single point in the Gromov boundary. Henceforth,  by a horoball we will mean a combinatorial horoball.

\begin{definition}(Definition 3.12 of \cite{GrovesManning})\label{cuspedspace}. Let $G$ be a finitely generated group and $\mathcal{H}=\{H_i\}$ be a finite collection of finitely generated subgroups. Let $S$ be a finite generating set of $G$ such that for each $i$, $S_i:=S\cap H_i$ generates $H_i$. Let $\Gamma(G,S)$ denote the Cayley graph of $G$ with generating set $S$. For each $i$, attach a horoball to the Cayley graph $\Gamma(H_i,S_i) \subseteq \Gamma(G,S)$. Similarly attach a horoball to each left coset $gH_i$ and denote it by $gH_i^{h}$. Union of $G$ and these attached horoballs with the induced path metric, say $d$, from $\Gamma(G,S)$ and horoballs, is called $cusped$ $space$ for the triple $(G,\mathcal{H},S)$ and denoted by $G^{h}$.
\end{definition}

It is a standard fact that $(G^{h},d)$ is a proper geodesic metric space and $G$ acts properly discontinuously on it. We denote the combinatorial horoball attached along the left coset $gH_i$ by $gH_i^h$ and metric on it by $d_{gH_i^h}$.
\begin{remark}
Let $G$ be a finitely generated group and $\mathcal{H}=\{H_i\}$ be a finite collection of finitely generated subgroups. If $S,S'$ are two finite generating sets of $G$ such that $S\cap H_i,S'\cap H_i$ also generates $H_i$ then there exists a quasi-isometry $\phi:\Gamma(G,S)\to\Gamma(G,S')$ such that $\phi(H_i)=H_i$. Let $\Gamma(G,S))^h$, $\Gamma(G,S'))^h$ denote the cusped space  on $G$ with respect to the generating sets $S,S'$ respectively.
The map $\phi$ induces a quasi-isometry $\phi^h:(\Gamma(G,S))^h\to (\Gamma(G,S'))^h$ (See Lemma 1.2.31 of \cite{PalThesis}). Thus, the contracting boundaries $\partial_c(\Gamma(G,S))^h$
and $\partial_c(\Gamma(G,S'))^h$ are homeomorphic to each other.

\end{remark}

\begin{definition}(Section 3.3 of \cite{GrovesManning})(Relatively hyperbolic group) $G$ is hyperbolic relative to a finite collection of finitely generated subgroup $\mathcal{H}=\{H_i\}$ if for some finite generating set $S$, cusped space $(G,\mathcal{H},S)$ is a hyperbolic space.\par
A finitely generated group $G$ is said to be \textit{relatively hyperbolic group}, if there exists some finite collection $\mathcal{H}$ of finitely generated subgroups of $G$ such that $G$ is hyperbolic relative to $\mathcal{H}$.
\end{definition}

\subsection{Some useful results}
\begin{proposition}(Lemma 4.4, \cite{cashen2017})\label{Keyprop1} Suppose
$\alpha$ is a continuous, $\rho$-contracting, $(L, A)$-quasi-geodesic and
$\beta$ is a continuous $(L, A)$-quasi-geodesic ray such that $d(\alpha_{0},
\beta_{0}) \leq k(\rho, L, A)$. If there are $r, s \in [0, \infty)$ such that $d(\alpha_{r}, \beta_{s}) \leq k(\rho, L,
A)$ then $d_{Haus}(\alpha[0,r],\beta[0,s]) \leq k'(\rho, L, A)$.
If $\alpha[0,\infty)$ and $\beta[0,\infty)$ are asymptotic then their Hausdorff
distance is at most $k'(\rho,L,A)$.
\end{proposition}

\begin{lemma}\label{CM 4.6}(Lemma 4.6,  \cite{cashen2017}) Let $\alpha$ be a $\rho$-contracting geodesic ray, and let $\beta$ be a continuous $(L,A)$-quasi-geodesic with $\alpha_{0}=\beta_{0}=o$. Given some $R$ and $J$, suppose there exists a point $x\in \alpha$ with $d(x, o)\geq R$ and $d(x, \beta)\leq J$. Let $y$ be the last point on the subsegment of $\alpha$ between $o$ and $x$ such that $d(y,\beta)\leq \kappa(\rho,L,A)$. Then there exists a constant $M\leq2$ and a function $\lambda(\phi,p,q)$ defined for sublinear function $\phi$, $p\geq1$ and $q\geq 0$ such that $\lambda$ is monotonically increasing in $p$ and $q$ and:
                                         $$d(o,y)\geq R-MJ-\lambda(\rho,L,A)$$
\end{lemma}

\begin{lemma}(Lemma 4.7, \cite{cashen2017})\label{Lemma 4.7}
Given a sublinear function $\rho$ and constants $L \geq 1$, $A \geq 0$ there
exist constants $L' \geq 1$ and $A'\geq 0$ such that if $\alpha$ is a $\rho$-contracting geodesic ray
or segment and $\beta$ is a continuous $(L,A)$-quasi-geodesic ray not asymptotic to $\alpha$
with $\alpha_0 = \beta_0 = o$, then we obtain a continuous $(L',A')$-quasi-geodesic by following $\alpha$
 backward until $\alpha_{s_0}$, then following a geodesic from $\beta$ to $\beta_{t_0}$, then following $\beta$,
where $\beta_{t_0}$ is the last point of $\beta$ at distance $\kappa(\rho,L,A)$ from $\alpha$, and where $\alpha_{s_0}$ is the
last point of $\alpha$ at distance $\kappa(\rho,L,A)$ from $\beta_{t_0}$ .
\end{lemma}

\section{Geometry of cusped space}
Let $G$ be a finitely generated group and $\mathcal{H}=\{H_1,H_2,...,H_n\}$ be a finite collection of finitely generated subgroups. Let $S$ be a finite generating set of $G$ such that for each $i$, $S_i:=S\cap H_i$ generates $H_i$.
Let $(G^{h},d)$ be the cusped space defined as in Definition \ref{cuspedspace}. Consider  the horoball $gH_i^h$ attached along the left coset $gH_i$ for $H_i\in \mathcal{H}$. Let $gh\in gH_i$ be a point and $e_i$ be the edge joining $(gh,i)$ and $(gh,i+1)$ for $i\geq0$. Define $\gamma:=\bigcup_{i\geq0}e_i$ . It is easy to check that $\gamma$ is a geodesic in $G^{h}$. We call $\gamma$ to be a \textit{vertical ray}.\par
We assume a vertical ray in each $H_i^h$ is contracting. Then we will prove that the limit set $\Lambda_c(gH_i^h)$ is a singleton set.

 \begin{lemma} \label{horoball contracting} Let $(G^{h},d)$ be a cusped space such that a vertical ray in each $H_i^h$ is contracting in $G^h$, then $H_i^h$ is Morse hence contracting in $G^h$. Conversely, if $H_i^h$ is contracting in $G^h$ then a vertical ray in each $H_i^h$ is contracting in $G^h$
     \end{lemma}
     \begin{proof}Given $k\geq1$ and $\epsilon\geq0$, let $\gamma$ be any $(k,\epsilon)$ quasi-geodesic lying outside $H_i^h$ with end point on $H_i\subseteq H^h$. It suffices to show that $\gamma$ lies in a bounded neighborhood of $H_i^h$ where the bound depends on $k$ and $\epsilon$ only.
 Since $H_i$ acts on $H_i^h$ by isometry, every vertical ray starting from some point of $H_i$ is $\rho_i$-contracting geodesic ray for some sublinear function $\rho_i$. Let $x$ and $y$ be end points of $\gamma$. For some number $n$, $2^{n-1}\leq d_{H_i}(x,y)\leq 2^n$. Consider points $(x,n),(y,n)\in H_i^h$. Since the path $\gamma$ lies outside $H_i^h$, point $x$ is a nearest point projection of $(x,n)$ onto $\gamma$. Let $\alpha$ be the vertical line from $(x,n)$ upto $x$ and $\beta$ be the vertical ray starting from $y$. Also let $e$ be an edge joining $(y,n)$ to $(x,n)$ then the concatenated path $\gamma':=e*\alpha* \gamma$ is a $(2k+1,\epsilon+1)$ quasi-geodesic joining points $(y,n)$ and $y$. Since $\beta$ is $\rho_i$-contracting, $\gamma'$ lies in $\kappa'(\rho_i,2k+1,\epsilon+1)$ neighborhood of $\beta$. Therefore $\gamma$ lies in $\kappa'(\rho_i,2k+1,\epsilon+1)$ neighborhood of $H_i^h$.\par
 Conversely, suppose $\gamma$ is a quasi-geodesic joining two points of a vertical geodesic ray $\alpha$ in $H_i^h$.
 Then, as $H_i^h$ is contracting, the portions of $\gamma$
  outside $H_i^h$ stay in a bounded neighborhood of $H_i^h$
  (the bound depends only up on the defining constants of the quasi-geodesic $\gamma$ and the contracting function of $H_i^h$). Taking the projection of $\gamma$ onto $H_i^h$, gives rise to a quasi-geodesic $\beta$ in $H_i^h$ such that $\gamma$ and $\beta$ lie in bounded neighborhood of each other. As $H_i^h$ is a hyperbolic metric space, $\beta$ lie in a bounded neighborhood of $\alpha$. Hence, $\gamma$ lie in a bounded neighborhood of $\alpha$.
     \end{proof}
  \noindent{\textbf{Note :}}  Henceforth, in this section, we consider the space $(G^{h},d)$ with the assumption that for each $i$, $H_i$ is of infinite index in $G$ and $H_i^{h}$ is a contracting subset of $G^{h}$. As $G$ acts on $G^h$ by isometry,  so $gH_i^h$ is contracting for all $g\in G$.

     \begin{lemma}\label{proper}
     For each $i$, $gH_i^{h}$ is uniformly properly embedded in $G^h$ i.e. for all $M>0$, there exists $N=N(M)$ such that if $x,y\in gH_i^h$ with $d(x,y)\leq M$ implies $d_{gH_i^h}(x,y)\leq N$.
     \end{lemma}
     \begin{proof} Let $H_i:=H$ and $\alpha:[0,l]\rightarrow G^h$ be a geodesic with $\alpha(0)=x$, $\alpha(l)=y$. Then $l_{G^h}(\alpha)\leq M$. $\alpha$ either lives in $G$ or in attached horoballs. Let $g_0H^h$, $g_1H^h$, ..., $g_nH^h$ be horoballs penetrated by $\alpha$ with $g_0H^h$=$gH^{h}$=$g_nH^h$. Partition $[0,l]$ by points $0=s_0\leq t_0< s_1<t_1<...<s_n\leq t_n=l$ such that $\alpha(t_0),\alpha(s_n)\in gH^h$, $\alpha_{[s_j,t_j]}$ is a geodesic in $g_jH^h$ and $\alpha_{[t_i,s_{i+1}]}$ is a geodesic in $G$ for $0\leq j\leq n$ and $0\leq i\leq n\mbox{-1}$.\par
    $\sum_{j=0}^{n}d_{g_jH^h}(\alpha_{s_j},\alpha_{t_j})$, $\sum_{i=0}^{n-1}d_{G}(\alpha_{t_i},\alpha_{s_{i+1}})\leq M$. There exists a number $N'=N'(M)$ such that $\sum_{j=0}^{n}d_{g_jH}(\alpha_{s_j},\alpha_{t_j})\leq N'2^M$. $d_{G}(\alpha_{t_0},\alpha_{s_n})$ $\leq$ $\sum_{j=1}^{n-1}d_{g_jH}(\alpha_{s_j},\alpha_{t_j})$+$\sum_{i=0}^{n-1}d_{G}(\alpha_{t_i},\alpha_{s_{i+1}})$ $\leq N'2^M+M$. Let $N_1:=N'2^M+M$. Then $d_{G}(\alpha_{t_0},\alpha_{s_n})\leq N_1$. As $H$ is finitely generated, it is easy to check that $gH$ are uniformly properly embedded in $G$. That gives a number $N_2=N_2(M)$ such that $d_{gH}(\alpha_{t_0},\alpha_{s_n})\leq N_2$. Hence,
    \begin{align*}
    d_{gH^h}(x,y)&\leq d_{gH^h}(\alpha_{s_0},\alpha_{t_0})+ d_{gH^h}(\alpha_{t_0},\alpha_{s_n})+d_{gH^h}(\alpha_{s_n},\alpha_{t_n})\\
                 &\leq M+ d_{gH}(\alpha_{t_0},\alpha_{s_n})+M\\
                 &\leq 2M+N_2.
    \end{align*}
    Take $N=2M+N_2$. This $N$ depends upon $M$ only.
    \end{proof}
    \begin{lemma}\label{qi lemma}
    For each $i$, inclusion $gH_i^h\hookrightarrow G^h$ is a quasi-isometric embedding.
    \end{lemma}
\begin{proof}
Let $x,y\in gH_i^h$. Clearly $d(x,y)\leq d_{gH_i^h}(x,y)$. Let $\alpha$ be a geodesic joining $x$ and $y$ in $G^h$. As $gH_i^h$ is contracting, by Theorem \ref{GIT}, there exists a sublinear function $\sigma$ and a constant $\kappa_{\sigma}$, such that $\alpha$ lies in $\kappa_{\sigma}+\frac{\sigma(\kappa_{\sigma})}{2}$ neighborhood of $gH_i^h$. Let $\kappa_{\sigma}':=\kappa_{\sigma}+\frac{\sigma(\kappa_{\sigma})}{2}$. Let $\alpha_s,\alpha_t\in \alpha$ such that $d(\alpha_s,\alpha_t)\leq 1$. Then there exist points $p_s,p_t\in gH_i^h\ $ such that $d(\alpha_s,p_s)\leq \kappa_{\sigma}'$ and  $d(\alpha_t,p_t)\leq \kappa_{\sigma}'$. By triangle inequality,
$d(p_s,p_t)\leq 2\kappa_{\sigma}'+1$.
Let $M=2\kappa_{\sigma}'+1$ and $C$ be the value of $N$ for $M$ in the Lemma \ref{proper}.
Then $d_{gH_i^h}(x,y)\leq C d(x,y)+C$. Therefore, the inclusion is $(C,C)$ quasi-isometric embedding.
\end{proof}

\begin{proposition} $\Lambda_c(gH_i^h)$ is a singleton set in $\partial_cG^h$.
\end{proposition}
\begin{proof}
The combinatorial horoball $H^h_i$ is a proper hyperbolic metric space whose Gromov boundary is a singleton set. As $H^h_i$ is quasi-isometrically embedded and contracting in $G^h$, by Corollary 6.2 of \cite{cashen2017}, $\Lambda_c(gH_i^h)$ is a singleton set in $\partial_cG^h$ for each $gH_i^h$.
\end{proof}
\begin{remark}\label{rem} Note that no two distinct combinatorial horoballs $g_iH^h_i$ and $g_jH^h_j$ have the same limit point in $\partial_cG^h$.
\end{remark}
 A finite collection of subgroups $\{H_1,...,H_n\}$ is said to be \textit{almost malnormal} if $\abs{gH_ig^{-1}\cap H_i}< \infty$ for
 $g\notin H_i$ and $\abs{gH_ig^{-1}\cap H_j}<\infty$
 for all $g\in G$ and $i\neq j$.

\begin{corollary}\label{malnormal} The finite collection
$\{H_1,...,H_n\}$ of subgroups is almost malnormal in $G$.
\end{corollary}
\begin{proof}
Suppose $\{H_1,...,H_n\}$ is not almost malnormal in $G$ then there exists an element $g\in G$ and two sequence of distinct elements $\{h_n\}$ and $\{h_n'\}$ in $H_i,H_j$ respectively such that $gh_n'g^{-1}=h_n$ for every $n$ and if $i=j$ then $g\notin H_i$.  Let $\Lambda_c(H_i^h)=\zeta_1$ and $\Lambda_c(gH_j^h)=\zeta_2$. By Remark \ref{rem}, $\zeta_1\neq\zeta_2$. After passing to a subsequence if necessary, we can assume that $\{h_n\}$ and $\{gh_n'\}$ converges to $\zeta_1$ and $\zeta_2$ respectively. Next we claim that the sequence $\{gh_n'\}$ converges to $\zeta_1$, which then gives $\zeta_1=\zeta_2$, a contradiction.\par
\underline{Claim:} $\{gh_n'\}$ converges to $\zeta_1$.\par
The proof is based on the idea that if a sequence $\{x_n\}$ converges to some point of the boundary and there is another sequence $\{y_n\}$ such that $\{d(x_n,y_n)\}$ is bounded then $y_n$ will converge to the same point of the boundary.\par
Let $\gamma_n$ be an arbitrary continuous $(L,A)$-quasi-geodesic joining $e$ to $gh_n'$. Also let $\delta_n$ be a geodesic joining $gh_n'$ and $h_n$. For given $\gamma_n$, define $\beta_n:=\gamma_n*\delta_n$. Let $\abs{g}=K$, then $\beta_n$ is continuous $(L,A+K)$-quasi-geodesic joining $e$ to $h_n$.\par
As $h_n\rightarrow \zeta_1$, given $r\geq1$, there exists number $N=N(r)$, such that for all $n\geq N$;
  \begin{equation} d\big(\beta_n,\alpha^{\zeta_1}\big([r,\infty)\big)\big) \leq \kappa(\rho_{\zeta_1},L,A+K)\end{equation}
  Here $\alpha^{\zeta_1}$ is a geodesic ray starting from $e$ and representing $\zeta_1$.\par
  From previous inequality, $d\big(\gamma_n,\alpha^{\zeta_1}\big([r,\infty)\big)\big) \leq \kappa(\rho_{\zeta_1},L,A+K)+K$. For fix $\zeta_1$ and $K$, let $\kappa_1(L,A):=\kappa(\rho_{\zeta_1},L,A+K)+K$. Given $r\geq1$, it suffices to consider only those $L,A$ such that $L^2,A\leq r/3$. By Lemma \ref{CM 4.6}, $\gamma_n$ comes $\kappa(\rho_{\zeta_1},L,A)$-close to $\alpha^{\zeta_1}$ outside ball around $e$ of radius $r-M\kappa_1(L,A)-\lambda(\rho_{\zeta_1},L,A)$. As $L^2,A\leq r/3$ and functions $\kappa_1$ and $\lambda$ is increasing in $L,A$, $r-M\kappa_1(L,A)-\lambda(\rho_{\zeta_1},L,A)\geq r-M\kappa_1(\sqrt{r/3},r/3)-\lambda(\rho_{\zeta_1},\sqrt{r/3},r/3)$.\par
  For fix $\zeta_1$ and $K$, let $u(r)=M\kappa_1(\sqrt{r/3},r/3)+\lambda(\rho_{\zeta_1},\sqrt{r/3},r/3)$. Then, given $r\geq1$, for all $n\geq N$ where  $N=N\big(r+u(r)\big)$ from (3.1), following holds;
  $$d\big(\gamma_n,\alpha^{\zeta_1}\big([r,\infty)\big)\big) \leq \kappa(\rho_{\zeta_1},L,A).$$
  For every continuous $(L,A)$-quasi-geodesic joining $e$ to $gh_n'$. That means $gh_n'\rightarrow \zeta_1$.
\end{proof}

Next lemma shows that $\partial_cG^h$ is metrizable. Cashen and Mackay proved that the contracting boundary of Cayley graphs are metrizable (See Corollary 8.6 of \cite{cashen2017}) and here we adapt their  proof on cusped graphs.
\begin{lemma} \label{metrizable} $\partial_cG^h$ is metrizable.
\end{lemma}
\begin{proof} By Proposition 5.12 and Proposition 5.13 of \cite{cashen2017}, $\partial_cG^h$ is Hausdorff and regular. Therefore, by Urysohn metrization theorem, it suffices to show that $\partial_cG^h$ is second countable.\par
 Let $\Lambda_{c}(H_i^h)$$=$$\eta_i$. Consider the set $\mathcal{A}=\{\eta_1,...\eta_n\}$. \par
\underline{Claim 1:} Given $\zeta\in \partial_cG^h$ and $r\geq1$, there exists $R'\geq1$ such that for all $R_2$$\geq$$R_1$$\geq$$ R'$ there exists $g\in G$ and $\eta_i\in A$ such that $\zeta \in U(g\eta_i,R_2)\subseteq U(g\eta_i,R_1)\subseteq U(\zeta,r)$.\par
If $\zeta$$=$$\Lambda_{c}(g''H_i^h)$ for some left coset $g''H_i$, then  $g''\eta_i$$=$$\zeta$ and the result trivially follows. Therefore we assume $\zeta$ is not the limit set of any $gH_i^h$. Now fix an element $\eta$ of $\mathcal{A}$. Let $\alpha^{\zeta}$ be a geodesic ray starting from $e$ and representing $\zeta$. Let $\alpha:$$=$$\alpha^{\zeta}$, then $\alpha$ does not eventually go into any horoball. Assume $\alpha$ is parameterized by it's arc length with $\alpha_0$$=$$e$, then there exists a increasing sequence $\{\sigma_n\}$ of natural numbers such that $\alpha_{\sigma_n}\in G$. Also, for some $g'\in G$, $g'\eta$$\neq$$\eta$. Let $g'\eta$$=$$\eta'$. Join $\eta$ to $\eta'$ by a geodesic $\beta$. $\beta$ passes through a group element.  Assume, without loss of generality $\beta_0$$=$$e$. Choose $\rho$ such that $\alpha$, $\beta_{[0,\infty)}$ and $\bar{\beta}_{[0,\infty)}$ are $\rho$-contracting.\par
For all large $n$, at most one of $\alpha_{\sigma_n}\beta_{[0,\infty)}$ and $\alpha_{\sigma_n}\bar{\beta}_{[0,\infty)}$ stays in closed $\kappa'(\rho,1,0)$ neighborhood of $\alpha_{[0,\sigma_n]}$ for distance greater than $2\kappa'(\rho,1,0)$ from $\alpha_{\sigma_n}$. To see this, suppose  for some number $n$ there exist points $x_n \in \alpha_{\sigma_n}\beta_{[0,\infty)}$, $y_n\in \alpha_{\sigma_n}\bar{\beta}_{[0,\infty)}$ and  $x_n',y_n'\in\alpha_{[0,\sigma_n]}$ such that $d(x_n,x_n')\leq\kappa'(\rho,1,0)$, $d(y_n,y_n')\leq\kappa'(\rho,1,0)$ and $d(\alpha_{\sigma_n},x_n)=d(\alpha_{\sigma_n},y_n)>2\kappa'(\rho,1,0)$. Since $\beta$ is geodesic, $d(x_n,y_n)> 4\kappa'(\rho,1,0)$. Now, using triangle inequality we get, $d(x_n',y_n')\leq 2\kappa'(\rho,1,0)$. Further application of triangle inequality gives $d(x_n,y_n)\leq 4\kappa'(\rho,1,0)$, a contradiction. \par
For each $n$, define $g_n$$:=$$\alpha_{\sigma_n}$ if $\alpha_{\sigma_n}\beta_{[0,\infty)}$ does not remain in $\kappa'(\rho,1,0)$ neighborhood of $\alpha_{[0,\sigma_n]}$ for distance greater than $2\kappa'(\rho,1,0)$ otherwise define $g_n$$:=$$\alpha_{\sigma_n}g'$. For fix $n$, join points $e$ and $g_n\beta_{m}$ by a geodesic $\nu^{m,n}$ for each integer $m\geq0$. For fix $n$, application of Arzela-Ascoli theorem gives a subsequence of $\{\nu^{m,n}\}$ converging to a geodesic ray $\nu^m$. By Lemma \ref{Clemma}, geodesics $\alpha_{[0,\sigma_n]}$ and $g_n\beta_{[0,m]}$ is contracting and the contraction function is independent of $m$ and $n$. Therefore, by Lemma \ref{almosttriangle}, geodesic rays $\nu^n$ are uniformly contracting. If $\alpha^{g_n\eta}$ is any arbitrary geodesic in $g_n\eta$, then $\alpha^{g_n\eta}$ and $\nu^n$ are asymptotic.\par
By Proposition \ref{Keyprop1}, there exists a sublinear function $\rho'$ such that $\alpha^{g_n\eta}$ are $\rho'$-contracting. Also, definition of $g_n$ ensures that there exists a constant $C\geq0$ independent of $n$ such that $\alpha^{g_n\eta}$ comes $\kappa(\rho,1,0)$ close to $\alpha_{[0,\sigma_n]}$ outside ball around $e$ of radius $\sigma_n$-$C$.
To see this let $u_n$ and $u_n'$ be respective last points on $\alpha^{g_n\eta}$ and $g_n\beta$, which is $\kappa(\rho,1,0)$-close to $\alpha_{[0,\sigma_n]}$.
Note that conditions on $g_n$ gives that $d(\alpha_{\sigma_n},u_n')\leq 2\kappa'(\rho,1,0)+d(e,g')$. Let $z_n\in \pi_{\alpha_{[0,\sigma_n]}}(u_n)$ and $z_n'\in \pi_{\alpha_{[0,\sigma_n]}}(u_n')$ be points.
The geodesic triangle $\triangle$ with sides $\alpha_{[0,\sigma_n]}$, $\alpha^{g_n\eta}$ and $g_n\beta$ is $\delta$-thin for some $\delta>0$ depending on the contraction functions of the geodesics $\alpha$ and $\beta$.
There exist points $v_n,v_v'$ where $v_n$ lie in the portion of
 $\alpha^{g_n\eta}$ between $u_n, {g_n\eta}$
 and $v_n'$ lie in the portion of $g_n\beta$
 between $u'_n, {g_n\eta}$
 such that the distances of $v_n,v_n'$
 to the geodesic $\alpha_{[0,\sigma_n]}$ is at most $\delta$ and  $d(v_n,v_n')\leq \delta$.
The portions of $\alpha^{g_n\eta}$ between $u_n, v_n$ and  $g_n\beta$ between $u'_n,v'_n$ lie outside $\kappa(\rho,1,0)$-neighborhood of $\alpha_{[0,\sigma_n]}$. Therefore, by Theorem \ref{QGIT}, the diameters of the nearest point projections of the mentioned portions are uniformly bounded. By using triangle inequality,
there exists $K\geq0$ depending only on the contraction functions of the geodesics $\alpha$ and $\beta$ such that $d(z_n,z_n')\leq K$ for any $n$.  Now take $C$ to be $K+2\kappa'(\rho,1,0)+d(e,g')+\kappa(\rho,1,0)$. This implies that $\alpha^{g_n\eta}_{[0,\sigma_n-C]}\subseteq \bar{N}_{2\kappa'(\rho,1,0)}\alpha_{[0,\sigma_n-C]}$.   \par
First we give a condition that implies $\zeta\in U(g_n\eta,R)$. Suppose $n$ is a number such that,
\begin{equation}\label{condition 1} \sigma_n\geq R+C+2M(\kappa'(\rho,\sqrt{R/3},R/3)+\kappa'(\rho,1,0))+\lambda(\rho,\sqrt{R/3},R/3)
\end{equation}
Suppose, $\gamma\in\zeta$ is a continuous $(L,A)$-quasi-geodesic. It suffices to consider $L^2,A<R/3$. $\gamma\subseteq \bar{N}_{\kappa'(\rho,L,A)}\alpha$, so there is a point $\gamma_a$ that is $2(\kappa'(\rho,\sqrt{R/3},R/3)+\kappa'(\rho,1,0))$-close to $\alpha^{g_n\eta}_{\sigma_n-C}$. By Lemma \ref{CM 4.6}, $\gamma$ comes $\kappa(\rho',L,A)$-close to $\alpha^{g_n\eta}$ outside the ball around $e$ of radius $\sigma_n-C-2M(\kappa'(\rho,\sqrt{R/3},R/3)+\kappa'(\rho,1,0))-\lambda(\rho,\sqrt{R/3},R/3)$. By, equation \ref{condition 1} this is at least $R$. Since $\gamma$ is arbitrary, $\zeta\in U(g_n\eta,R)$.  \par
Next, we give a condition that implies $U(g_n\eta,R)\subseteq U(\zeta,r)$. Suppose $n$ is a number such that:
\begin{equation} \label{condition 2}\sigma_n-C\geq R\geq r+2M(\kappa'(\rho',\sqrt{r/3},r/3)+\kappa'(\rho,1,0))+\lambda(\rho,\sqrt{r/3},r/3)
\end{equation}
Suppose, $\gamma\in\zeta$ is a continuous $(L,A)$-quasi-geodesic such that $\gamma\in U(g_n\zeta,R)$. It suffices to consider $L^2,A<R/3$. By definition, $\gamma$ comes $\kappa(\rho',L,A)$-close to $\alpha^{g_n\eta}$ outside $N_Re$, so some point $\gamma_b$ is $2\kappa'(\rho',L,A)$-close to $\alpha_R^{g_n\eta}$, which implies that $d(\gamma_b,\alpha_R^{g_n\eta})\leq 2\kappa'(\rho',L,A)+2\kappa'(\rho,1,0)$. Now apply Lemma \ref{CM 4.6} to see that $\gamma$ comes $\kappa(\rho,L,A)$-close to $\alpha$ outside the ball around $e$ of radius at least $R-2M(\kappa'(\rho',L,A)+\kappa'(\rho,1,0))-\lambda(\rho,L,A)$, which is at least $r$, by equation \ref{condition 2}. Thus $U(g_n\eta,R)\subseteq U(\zeta,r)$. \par
 Using these two conditions we will prove the claim. Functions $\rho$ and $\rho'$ are determined by $\zeta$ and $\eta$. Given these and any $r\geq1$, define $R'$ to be right hand side of equation \ref{condition 2}. Given any $R_2\geq R_1\geq R'$, it suffices to define $g:=g_n$ for any large $n$ satisfying equation \ref{condition 1} for $R=R_2$ and equation\ref{condition 1} for $R=R_1$ simultaneously.\par
 \underline{Claim 2:} $\partial_cG^h$ is second countable.\par
 For given $\zeta\in \partial_cG^h$ and $r\geq 1$, Corollary 5.10 of \cite{cashen2017} gives an open set $U$ and number $\phi(\rho_{\zeta},r)\geq1$ such that $U(\zeta,\phi(\rho_{\zeta},r))\subseteq U\subseteq U(\zeta,r)$. For each $g$ and $n$ and $\eta_i$, choose open set $U_{g,n}^i$ such that $U(g\eta_i,\phi(\rho_{g\eta_i},n))\subseteq U_{g,n}^i\subseteq U(g\eta_i,n)$.\par
 Let U be a non-empty set $\zeta\in U$, a point. Also assume $\zeta$ is not the limit set of any horoball and $\eta$ as in proof of previous claim i.e. $\eta=\eta_i$  for some $i$. Denote $U_{g,n}^i$ by $U_{g,n}$. By definition, there exists some $r\geq1$, such that $U(\zeta,r)\subseteq U$. Let $R'$ be a constant from \underline{Claim 1} for $\zeta$ and $r$ and $R'\leq R_1$ be a natural number. As in the proof of \underline{Claim 1}, $g_n\eta$ is $\rho'$-contracting for every $n$. Define $R_2:=\phi(\rho',R_1)\geq \phi(\rho_{g_n\eta},R_1)$. Combining \underline{Claim 1} and definition of  $U_{g,n}$ we get;
 $$\zeta\in U(g\eta,R_2)\subseteq U_{g,R_1}\subseteq U(g\eta,R_1)\subseteq U(\zeta,r)$$
 Let $\mathcal{U}^i:=\{U_{g,n}^i| g\in G, n\in \mathbb{N}\}$. Then $\bigcup \mathcal{U}^i$  is a countable basis for $\partial_cG^h$.

 \end{proof}

In \cite{cashen2017}, Cashen and Mackay proved that if $\partial_cG$ is compact then every geodesic ray in $G$ is contracting (Theorem 10.1, 5 $\implies$ 6, of \cite{cashen2017}). Using the same idea, we prove the following lemma.

      \begin{lemma}\label{ContRay}
     If $\partial_cG^{h}$ is  compact then every geodesic ray in $G^{h}$ is contracting.
     \end{lemma}
     \begin{proof}
     It is sufficient to show that all geodesic rays starting from $e$ are contracting. As each $H_i$ is of infinite index, there are infinitely many points in $\partial_{c}G^{h}$. Let $\eta$ be the limit point of a horoball. There exists $g$ in $G$ such that $g\eta \neq \eta$, call $g\eta$ to be $\eta'$. Let $\beta$ be a bi infinite geodesic joining $\eta'$ to $\eta$. Clearly $\beta$ passes through a group element so without loss of generality take $\beta_{0}=e$.\par
      Let $\alpha$ be an arbitrary geodesic ray starting from $e$. For some sublinear function $\rho$, geodesic rays $\beta_{[0,\infty)}$,$\bar{\beta}_{[0,\infty)}$, $\alpha$ and all subsegments of these geodesic rays are $\rho$-contracting. Note that $\alpha$ either eventually goes in some horoball or it does not. In the former case it is readily contracting as each $gH_i^{h}$ is contracting subset and hence the tail of $\alpha$ is contracting in $G^{h}$. So we are assuming that $\alpha$ does not eventually go in any horoball. That means there exists an increasing sequence of natural numbers $\{\sigma''_n\}$ such that sequence $\{\alpha_{\sigma''_n}\}\subseteq G$ . For all large $n$, initial segment of length greater than $2\kappa'(\rho,1,0)$ of  both $\alpha_{\sigma''_n}\beta_{[0,\infty)}$ and $\alpha_{\sigma''_n}\bar{\beta}_{[0,\infty)}$ will not come in closed $\kappa'(\rho,1,0)$ neighborhood of $\alpha_{[0,\sigma''_n]}$. To see this, let $n>>1$ such that there exists points $x_n,y_n$ on $\alpha_{\sigma''_n}\beta_{[0,\infty)}$, $\alpha_{\sigma''_n}\bar{\beta}_{[0,\infty)}$ respectively and two points $x_n',y_n'$ on $\alpha_{[0,\sigma''_n]}$ with $d(x_n,x_n')$,$d(y_n,y_n')\leq \kappa'(\rho,1,0)$ and $d(\alpha_{\sigma_n''},x_n)=d(\alpha_{\sigma_n''},y_n)>2\kappa'(\rho,1,0)$. Since $\beta$ is geodesic, $d(x_n,y_n)> 4\kappa'(\rho,1,0)$. Now, using triangle inequality we get, $d(x_n',y_n')\leq 2\kappa'(\rho,1,0)$. Further application of triangle inequality gives $d(x_n,y_n)\leq 4\kappa'(\rho,1,0)$, a contradiction. Above means, after possibly exchanging $\beta$ with $\bar{\beta}$, there exists an increasing sequence(subsequence of $\{\sigma''_n\}$) of natural numbers $\{\sigma'_n\}$ such that initial segment of length greater than $2\kappa'(\rho,1,0)$ of $\alpha_{\sigma'_n}\beta_{[0,\infty)}$ does not come in $\kappa'(\rho,1,0)$ neighborhood of $\alpha_{[0,\sigma'_n]}$. Let $\alpha_{\sigma'_n}\beta(t_n)$ be the last point of $\alpha_{\sigma'_n}\beta$ at distance $\kappa(\rho,1,0)$ from $\alpha_{[0,\sigma'_n]}$ and  $\bar{\alpha}_{[0,\sigma'_n]}(s_n)$  be last point of $\bar{\alpha}_{[0,\sigma'_n]}$ at distance $\kappa(\rho,1,0)$ from $\alpha_{\sigma'_n}\beta({t_n})$. Let $\mu_n$ be a geodesic joining $\alpha_{\sigma'_n}\beta(t_n)$ and $\bar{\alpha}_{[0,\sigma'_n]}(s_n)$. From Lemma \ref{Lemma 4.7}, concatenated paths $\alpha_{[0,s_n]}*\mu_n*\alpha_{\sigma'_n}\beta([t_n,\infty))$ is $(L',A')$-quasi-geodesic for all $n$. Note that $t_n\leq 2\kappa'(\rho,1,0)$ and length of $\mu_n\leq \kappa(\rho,1,0)$. This gives length of $\alpha_{[s_n,\sigma'_n]}*\alpha_{\sigma'_n}\beta([0,t_n])$ is bounded by $2\kappa'(\rho,1,0)+\kappa(\rho,1,0)$ $\leq 3\kappa'(\rho,1,0)$. Hence there exists $(L,A)$ such that $\alpha_{[o,\sigma'_n]}*\alpha_{\sigma'_n}\beta$ is $(L,A)$-quasi-geodesic for all $n$.\par

      Since $\partial_{c}G^{h}$ is compact and  metrizable (Lemma \ref{metrizable}),  $\partial_{c}G^{h}$ is  sequentially compact. Therefore the sequence $\{\alpha_{\sigma'_n}\beta_{\infty}\}$ has a convergent subsequence. This gives a subsequence $\{\sigma_n\}$ of $\{\sigma'_n\}$ such that $\{\alpha_{\sigma_n}\beta_{\infty}\}$ converges to $\zeta \in \partial_{c}G^{h}$. By definition of convergence, for given $r> 3L^{2},3A$, quasi geodesic $\alpha_{[0,\sigma_n]}*\alpha_{\sigma_n}\beta_{[0,\infty)}$ comes $\kappa(\rho_{\zeta},L,A)$ close to $\alpha^{\zeta}$ outside $r$ ball around $e$, for all sufficiently large $n$. Therefore, for any given $r\geq1$ initial segment of $\alpha_{[0,\sigma_n]}*\alpha_{\sigma_n}\beta_{[0,\infty)}$ of length at least $r$ contained in $\kappa'(\rho_{\zeta},L,A)$ neighborhood of $\alpha^{\zeta}$, for all large $n$. Thus, $\alpha$ is in $\kappa'(\rho_{\zeta},L,A)$ neighborhood of $\alpha^{\zeta}$ hence asymptotic to $\alpha^{\zeta}$ which implies that $\alpha$ is contracting.
     \end{proof}

The following lemma  will be used to prove the sequential compactness of $G^{h}\cup \partial_cG^h$ provided  that $\partial_cG^h$ is compact (See Proposition \ref{compactf}).
 \begin{lemma}\label{building quasi-geodesic} Let $\beta$ be a contracting $(3,0)$-quasi-geodesic ray
 such that $\beta$ is concatenation of a geodesic $[o,p]$ and a geodesic ray $\nu$ starting from point $p$. We parametrize $\beta:[0,\infty)\to X$ by arc length where $\beta(0)=o$. Consider a point $q\in\beta ([0,\infty))$ which lies in $\nu$. Let $\gamma$ be any continuous $(L,A)$-quasi-geodesic joining $\beta(0)$$=$$o$ and $q$. Then there exists a $(2L+1,A)$-quasi-geodesic ray of the form $\gamma|_{[0,u]}*\mu$, where $\gamma(u)\in\gamma$ and $\mu$
 is a geodesic ray starting from $\gamma(u)$, such that $\gamma|_{[0,u]}*\mu$ is asymptotic to $\beta$.\end{lemma}
\begin{proof}
 For all $n$, there exist points $\gamma(t_n)\in\gamma$
 and $\beta(s_n)\in\beta([n,\infty))$ such that $d(\gamma,\beta|_{[n,\infty)})=d(\gamma(t_n),\beta(s_n))$.
 If $\mu_n$ is a geodesic between $\gamma(t_n)$ and $\beta(t_n)$
 then $\gamma|_{[0,t_n]}*\mu_n$ is a $(2L+1,A)$-quasi-geodesic. An application of Arzela-Ascoli theorem tells that
 there exists a quasi-geodesic ray of the form $\gamma|_{[0,u]}*\mu$ where $\gamma(u)\in\gamma$ and $\mu$
 is a geodesic ray starting from $\gamma(u)$. Each $\gamma|_{[0,t_n]}*\mu_n$ lie $R$-neighborhood of $\beta$, where $R$
 depends on $L,A$ and the contraction function of $\beta$. Thus, $\gamma|_{[0,u]}*\mu$ also lie in the $R$-neighborhood
 of $\beta$ and hence it is asymptotic to $\beta$.  \end{proof}

\begin{proposition}\label{compactf}
If $\partial_cG^{h}$ is compact, then $G^{h}\cup \partial_cG^h$ is sequentially compact.
\end{proposition}
\begin{proof}
Given a sequence of points $\{x_n\}$, if it has a subsequence in $\partial_cG^h$ then we are through as $\partial_cG^h$ is sequentially compact (Lemma \ref{metrizable}). So, we assume $\{x_n\}$ lie in $G^h$. If $\{d(e,x_n)\}$ is bounded then we take a subsequence of it converging to it's limit point in $G^h$. Therefore we are left with the case that the sequence $\{d(e,x_n)\}$ is unbounded. Sequence $\{x_n\}$ has a subsequence either in $G$ or in horoballs.\par
\underline{Case 1:} $\{x_n\}$ has a subsequence in $G$. \par
Without loss of generality assume $\{x_n\}\subseteq G$. Let $\alpha$ be a $\rho$-contracting bi-infinite geodesic passing through $e$. Such geodesic exists as $\abs{\partial_cG^h}\geq 2$ and every bi-infinite geodesic passes through an element of $G$. Here, we are assuming that any sub path of $\alpha$ is also $\rho$-contracting. Define $\alpha_n$$:=$$x_n\alpha$. Each $\alpha_n$ is $\rho$-contracting and $x_n\in\alpha_n$. Let $a_n:=\alpha_n(+\infty)$, $b_n:=\alpha_n(-\infty)$ and $p_n$ be a nearest point projection of $e$ onto $\alpha_n$. Since $\partial_cG^h$ is sequentially compact, there exists a subsequence $\{a_{n_k}\}$ of $\{a_n\}$ converging to a point $a$ in $\partial_cG^h$. Also, without loss of generality, we assume that $\{x_{n_k}\}\in [p_{n_k},a_{n_k})$, sub path of $\alpha_{n_k}$ starting from $p_{n_k}$ and pointing in $a_{n_k}$ direction. Our goal is to show that $\{x_{n_k}\}$ converges to $a$.\par
For given constants $L$ and $A$, let $\gamma$ be an arbitrary continuous $(L,A)$ quasi-geodesic joining $e$ to $x_{n_k}$. By  Lemma \ref{building quasi-geodesic}, there exists a geodesic ray $\mu$ starting from a point on $\gamma$ and asymptotic to $a_{n_k}$. Moreover, if starting point of $\mu$ separates $\gamma$ into two segments $\gamma_1$ and $\gamma_2$ then concatenated paths $\gamma_1*\mu$ and $\gamma_2*\mu$ will be  $(2L+1,A)$ quasi-geodesic rays. As $a_{n_k}$ converges to $a$, given $r\geq1$, for all large $k$ there exists a point $y_{n_k}$ on $\gamma_1*\mu$ such that $d\big(y_{n_k},\alpha^a[r,\infty)\big)\leq \kappa(\rho_a,2L+1,A)$. If point $y_{n_k}\in \gamma_1$, then $d\big(\gamma,\alpha^a[r,\infty)\big)\leq \kappa(\rho_a,2L+1,A)$. Now, let $y_{n_k}\in \mu$. Note that $y_{n_k}\in \gamma_2*\mu$, which is a $(2L+1,A)$ quasi-geodesic ray starting from $x_{n_k}$ and asymptotic to sub path $[x_{n_k},a_{n_k})$ of $\alpha_n$. Since sub path $[x_{n_k},a_{n_k})$ is $\rho$-contracting, we get a constant $\delta'=\delta'(\rho,L,A)$ and a point $z_{n_k}\in [x_{n_k},a_{n_k}) $ such that $d(y_{n_k},z_{n_k})\leq \delta'$. Therefore, $d(z_{n_k},\alpha^a[r,\infty))\leq \delta'+\kappa(\rho,2L+1,A)$. Path $\beta_{n_k}:=[e,p_{n_k}]*[p_{n_k},z_{n_k}]$ is a $(3,0)$ quasi-geodesic. Let $\nu_{n_k}$ is a geodesic joining $z_{n_k}$ to a point of $\alpha^a[r,\infty)$ of length at most $\delta'+\kappa(\rho,2L+1,A)$. Concatenated path $\beta_{n_k}*\nu_{n_k}$ is a $(L',A')$ quasi-geodesic whose starting point is $e$ and end point on $\alpha^a[r,\infty)$ where $L'$ and $A'$ depends only on $L$, $A$, $\rho$ and $\rho_a$. Since geodesic ray $\alpha^a$ is $\rho_a$ contracting, there exists a constant $\delta''=\delta''(\rho_a,L,A)$ such that $d\big(x_{n_k},\alpha^a\big)\leq \delta''$. For all sufficiently large $k$, $x_{n_k}$ is $\delta''$ close to $\alpha^a[r,\infty)$. Let $\delta(L,A):=\mbox{max}\{\delta'',\kappa(\rho_a,2L+1,A)\}$. Hence, for all large $k$, we get $d\big(\gamma,\alpha^a[r,\infty)\big)\leq \delta(L,A)$.\par
Note that in the definition of convergence, given $r\geq1$, it suffices to consider only those $L,A$ such that $L\leq \sqrt{r/3}$ and $A\leq r/3$. Define $J:=\delta(\sqrt{r/3},r/3)$. From previous paragraph, for all large $k$, any $(L,A)$ quasi-geodesic joining $e$ and $x_{n_k}$ with $L\leq \sqrt{r/3}$ and $A\leq r/3$ is $J$-close to $\alpha^a[r,\infty)$. An application of Lemma \ref{CM 4.6} then gives that for all large $k$, $d\big(\gamma,\alpha^a[r,\infty)\big)\leq \kappa(\rho_a,L,A)$. Hence, $x_{n_k}\rightarrow a$.\par
\underline{Case 2:} $\{x_n\}$ has a subsequence in horoballs.\par
If that subsequence eventually goes in some horoball then we are done as we have taken horoballs to be a Morse subset in $G^h$. Therefore, without loss of generality, we assume that $\{x_{n}\}$ lies on distinct horoballs. Let $x_n$ lie in the horoball $B_n$ and $\Lambda(B_n)=a_n$. Also let $x_n'$ be the point of the coset corresponding to $B_n$ beneath the point $x_n$. Let $\theta_n$ be the vertical ray(a geodesic in $G^h$) starting from $x_n'$. Path $\theta_n$ passes through $x_n$, asymptotic to the point $a_n$ in the boundary and they all are uniformly contracting. Let $p_n$ be a nearest point projection of $e$ onto $\theta_n$. Clearly $d(e,x_n')$ is unbounded. By Case 1, there exists a subsequence $\{x_{n_k}'\}$ of $\{x_n'\}$ such that $x_{n_k}'\rightarrow b$ in $\partial_cG^h$. Also, since $\partial_cG^h$ is sequentially compact, the sequence $\{a_{n_k}\}$ further has a subsequence converging to a point in $a\in\partial_cG^h$. We assume without loss of generality that $a_{n_k}\rightarrow a$. Sequence $\{x_{n_k}\}$ has a subsequence lying either on $[p_{n_k},a_{n_k})$ or on $[p_{n_k},x_{n_k}']$. Similar to the proof of case 1, one can see that the subsequence converges to either $a$ or $b$. Moreover $a=b$, as if $a\neq b$, then by Theorem \ref{Mtheorem}, $\{d(e,p_{n_k})\}$ is a bounded sequence which is not true.
\end{proof}

\begin{remark}
Let $\eta\in \partial_cG^h$. Consider a geodesic $\alpha^{\eta}$ representing $\eta$ starting from $e$. Define $x_n:=\alpha^{\eta}(n)$. $\{x_n\}\subseteq G^h$ and converges to $\eta$. That shows $G^h$ is dense in $\bar{G^h}$. Also $G^h$ is an open set of $\bar{G^h}$.  By Proposition \ref{compactf}, $\bar{G^h}$ is a compactification of $G^h$.
\end{remark}

\section{Visual size of combinatorial horoballs}
In \cite{palpandey}, we proved that in a proper geodesic metric space if the end points of a sequence of bi-infinite geodesics converge to two distinct points in the contracting boundary then the sequence of geodesic passes through a bounded set (Theorem 3.13 of \cite{palpandey}). The same statement holds true if we take bi-infinite continuous $(P,\epsilon)$-quasi-geodesics instead of bi-infinite geodesics and the proof of it is the same as that of Theorem 3.13 in \cite{palpandey}. For the sake of completeness, we include the proof of above statement for a sequence of $(P,\epsilon)$-quasi-geodesics.
\begin{theorem}\label{Mtheorem}(Theorem 3.13 of \cite{palpandey})
  Let $P\geq 1$ and $\epsilon\geq 0$. Let $X$ be a proper geodesic metric space with ${\partial_{c}X}\neq\phi$.
Consider a sequence of continuous $(P,\epsilon)$-quasi-geodesics $\gamma_{n}$ with endpoints $(x_{n},y_{n})$ in
$X\cup \partial_{c}X$. Suppose $x_{n}\rightarrow \zeta$ and $y_{n}\rightarrow \eta$ in $X\cup \partial_{c}X$, with $\zeta\neq \eta$. Then $\gamma_{n}$ passes through an uniformly bounded set.
    \end{theorem}
     \begin{proof}
    If either $\zeta$ or $\eta$ belongs to $X$, then the statement is easy to prove. So, we assume $\zeta,\eta\in \partial_{c}X $. \par
    Let $k:=\mbox{max}\{\kappa(\rho_{\zeta},2P+1,\epsilon),\kappa(\rho_{\eta},2P+1,\epsilon)\}$. Fix a base point $o$ in $X$.
For points $x,y\in X\cup \partial_{c}X$, we denote $[x,y]$ by a geodesic segment, geodesic ray or bi-infinite geodesic joining $x$ and $y$, depending on whether $x$ and $y$, none, exactly one or both lies in $\partial_{c}X$.
If $\gamma$ is any parameterized path and $x,y \in \gamma$ then $[x,y]_{\gamma}$ stands for segment of $\gamma$ between $x$ and $y$.\par
    Let $\gamma$ be a bi-infinite geodesic joining $\zeta$ and $\eta$.
    Let $p_{n}$ and $p$ be the nearest point projections from $o$ to $\gamma_{n}$ and $\gamma$ respectively. We will prove that the
    sequence $\{d(o,p_{n})\}$ is bounded.\\
    Let $\alpha_{n}:=[o,p_{n}]\cup [p_{n},x_{n}]_{\gamma_{n}}$, $\alpha:=[o,p]\cup [p,\zeta]_{\gamma}$, $\alpha'_{n}:=[o,p_{n}]\cup [p_{n},y_{n}]_{\gamma_{n}}$,
    and $\alpha':=[o,p]\cup[p,\eta]_{\gamma}$. Paths $\alpha_{n}$ and $\alpha'_{n}$,  are continuous $(2P+1,\epsilon)$-quasi-geodesics. Paths $\alpha$ and $\alpha'$ are $(3,0)$-quasi-geodesics. Starting from $o$, we parameterize them by it's arc length . Since $x_{n}\rightarrow \zeta,y_{n}\rightarrow \eta$, by definition of convergence \eqref{eqn:key1}, \eqref{eqn:key2} and definition \ref{bordif}, for any $r\geq1$ there exists $N=N(r)$ such that for all $n\geq N$
\begin{align*}\label{CM eqn}
    d\big(\alpha_{n},\alpha^{\zeta}([r,\infty))\big)\leq k\ \  \&\ \ d\big(\alpha'_{n},\alpha^{\eta}([r,\infty))\big)\leq k ,
   \end{align*}
   where $\alpha^{\zeta}$ and $\alpha^{\eta}$ are geodesic rays from $o$ representing $\zeta$ and $\eta$ respectively.
   As $\alpha,\alpha'$ are $(3,0)$-quasi-geodesics, by Proposition \ref{Keyprop1},
   $\alpha^{\zeta}$ and $\alpha$, $\alpha^{\eta}$ and $\alpha'$ lie in a bounded Hausdorff distance, say $M$.
   Since $\alpha, \alpha'$ are $(3,0)$-quasi-geodesics parameterized by it's arc length with $\alpha(0)=o=\alpha'(0)$, we get following
   $$d(\alpha_n,\alpha([\frac{\mbox{(r-M)}}{3},\infty))\leq k+M,d(\alpha'_n,\alpha'([\frac{\mbox{(r-M)}}{3},\infty))\leq k+M.$$
   Let $k_1=k+M$ and $r_1=\frac{(r-M)}{3}$. By Lemma \ref{Clemma} and Theorem \ref{Cont} there exists a Morse gauge say $N_{\gamma}$ such that every subsegment of $\gamma$ is $N_{\gamma}$-Morse. We choose $r$ large such that $r_1>2k_1+3d(o,p)+N_{\gamma}(2P+1,2k_1+\epsilon)+1$.
  For such $r$, there exists a number $N=N(r)>0$ such that for any $n\geq N$, we have the following:
  \begin{equation} \label{closeness}
    d\big(\alpha_{n},\alpha ([r_1,\infty))\big)\leq k_1\ \&\
d\big(\alpha'_{n},\alpha'([r_1,\infty))\big)\leq k_1
\end{equation}
Thus, for $n\geq N$, there exist $u_n\in\alpha_n,v_n\in\alpha ([r_1,\infty)),u'_n\in\alpha'_n,v'_n\in\alpha'([r_1,\infty))$ such that $d(u_n,v_n)\leq k_1$ and $d(v_n',v_n')\leq k_1$.\par
   \textbf{Case 1:} Suppose $u_n\in[p_{n},x_{n}]_{\gamma_{n}}$ and $u_n'\in [p_{n},y_{n}]_{\gamma_{n}}$.
Path $[v_{n}',u_{n}']\cup [u'_{n},u_{n}]_{\gamma_{n}}\cup[u_{n},v_{n}]$ is $(P,2k_1+\epsilon)$-quasi-geodesic joining $v_{n}'$ and $v_{n}$.
   Path $[v_{n}',v_{n}]_{\gamma}$, being a subsegment of $\gamma$, is $N_{\gamma}$-Morse. Therefore Hausdorff distance between $[v_{n}',u_{n}']\cup [u'_{n},u_{n}]_{\gamma_{n}}\cup[u_{n},v_{n}]$ and $[v_{n}',v_{n}]_{\gamma}$ is  bounded, say $M_1$. Note that $M_1$ is independent of $n$. That gives a point $z_{n}\in [u_{n}',u_{n}]_{\gamma_n}$ such that $d(p,z_{n})\leq M_1$. As $p_{n}$ is a nearest point projection of $o$ on $\gamma_{n}$, $d(o,p_{n})\leq M_1+d(o,p)$. \par
   \textbf{Case 2:} Suppose $u_{n}\in [o,p_{n}]$ and $u_{n}'\in [p_{n},y_{n}]_{\alpha_{n}'}$ for some number $n$. Consider $(2P+1,\epsilon)$-quasi-geodesic path $[u'_{n},p_{n}]_{\gamma_{n}}\cup [p_{n},u_{n}]_{\alpha'_{n}}$. As $d(u_{n},v_{n})\leq k_1$ and $d(u_{n}',v_{n}')\leq k_1$, the path $[v'_{n},u'_{n}]\cup [u'_{n},p_{n}]_{\gamma_{n}}\cup [p_{n},u_{n}]_{\alpha'_{n}}\cup [u_{n},v_{n}]$ is a $(2P+1,2k_1+\epsilon)$-quasi-geodesic.
     Now $d(p,v_{n})> r_1 - d(o,p)$ and hence $d(v_{n},o)>r_1-2d(o,p)$. As $d(u_{n},v_{n})\leq k_1$, therefore $ d(o,u_{n})> r_1-k_1-2d(o,p)$.
     This implies that $(2P+1,\epsilon)$-quasi-geodesic path  $[u'_{n},p_{n}]_{\gamma_{n}}\cup [p_{n},u_{n}]_{\alpha'_{n}}$ lie outside the ball $\bar{B}(p; r_1-k_1-3d(o,p))$. Again, as $d(u_{n},v_{n})\leq k_1$ and $d(u_{n}',v_{n}')\leq k_1$, the path
    $[v'_{n},u'_{n}]\cup [u'_{n},p_{n}]_{\gamma_{n}}\cup [p_{n},u_{n}]_{\alpha'_{n}}\cup [u_{n},v_{n}]$ lie outside the ball
    $B(p; r_1-2k_1-3d(o,p))$. This says that the subsegment $\bar{B}(p;r_1-2k_1-3d(o,p))\cap \gamma$ is not $N_{\gamma}$-Morse, a contradiction. Thus, Case 2 is not possible. \par
    \textbf{Case 3:} Now assume that $u_{n}\in [o,p_{n}]$ and $u_{n}'\in [o,p_{n}]$ for some number $n$. Then,
    following the same argument as in the first part of case 2, we get  $d(o,u_{n}),d(0,u_{n}')> r_1-k_1-2d(o,p)$.
    Then path $[v'_{n},u'_{n}]\cup [u'_{n},u_{n}]_{\alpha_{n}'}\cup [u_{n},v_{n}]$, which are $(1,2k_{1})$-quasi-geodesic,
    lies outside the ball $\bar{B}(p; r_1-2k_1-3d(o,p))$ which says that segment $\bar{B}(p;r_1-2k_1-3d(o,p))\cap \gamma$ is not $N_{\gamma}$-Morse, a contradiction.\par
    Thus, for all $n\geq N$, $u_n\in[p_{n},x_{n}]_{\gamma_{n}}$ and $u_n'\in [p_{n},y_{n}]_{\gamma_{n}}$ and we return to the Case 1.
     \end{proof}

   \begin{definition}(Visual size of a combinatorial horoball)\label{visual size} Fix a combinatorial horoball $gH^h$. Let $p\in G^h\setminus gH^h$ a point. Let
 $V_p(gH^h)$:=$\{x\in gH |$ there exists a continuous $(3,0)$-quasi-geodesic $\gamma$ in $G^h$ starting from $p$ such that $\gamma\cap gH^h=\{x\}$  $\}$.
We call $diam(V_p(gH^h))$ to be $visual$ $size$ of the combinatorial horoball $gH^h$ with respect to the point $p$. The combinatorial horoball $gH^h$ is said to have bounded visual size if $V_p(gH^h)$ is uniformly bounded for all $p\in G^h\setminus gH^h$.
   \end{definition}
   It is natural to define the visual size of a combinatorial horoball using geodesics but we will be dealing with continuous $(3,0)$-quasi-geodesics to prove our results, so in the Definition \ref{visual size} we have taken $(3,0)$-quasi-geodesics instead of geodesics.

\begin{proposition}\label{VDprop1}
Let $p\in G^h\setminus gH^h$, where $gH^h$ is the combinatorial horoball attached to the left coset $gH$. Then $diam(V_p(gH^h))$ is bounded.
\end{proposition}
\begin{proof}
Let $o$ be a base point. Suppose $diam(V_p(gH^h)$ is unbounded, then there exists a sequence of points $\{x_n\}$ in $V_p(gH^h)\subseteq gH$ such that $d(o,x_n)\rightarrow \infty$. As $x_n\in V_p(gH^h)$, there exists a $(3,0)$-quasi-geodesic $\gamma_n$ starting from $p$ such that $\gamma_n\cap gH^h=\{x_n\}$. Let $\Lambda(gH^h)=\zeta$. Then $\{x_n\}$ has a subsequence converging to $\zeta$. Without loss of generality, let us assume $x_n\rightarrow \zeta$. \par
Let $q_n$ be a nearest point projection of $o$ onto $\gamma_n$. Parameterize each $\gamma_n$ by it's arc length with $\gamma_n(0)=p$. Let $\beta_n:=[o,q_n]*\gamma_n|_{[q_n,x_n]}$, where $\gamma_n|_{[q_n,x_n]}$ is restriction of $\gamma_n$ from $q_n$ until $x_n$. Clearly $\beta_n$ is $(7,0)$-quasi-geodesic starting from $o$. By definition of convergence, given $r\geq1$, there exists $N=N(r)$ such that for all $n\geq N$,
    $d\big(\beta_{n},\alpha^{\zeta}([r,\infty))\big)\leq$ $\kappa(\rho_{\zeta},7,0)$. Here $\alpha^{\zeta}$ is a geodesic representative of $\zeta$ starting from $o$. Since $d(o,q_n)\leq d(o,p)$, for all large $r$, points of $\beta_n$ which are $\kappa(\rho_{\zeta},7,0)$ close to $\alpha^{\zeta}([r,\infty))$ lies on $\gamma_n$. Note that $\alpha^{\zeta}$ eventually goes in $gH^h$ which gives a number $r\geq1$ such that for all $r'\geq r$, closed ball $\bar{B}\big(\alpha^{\zeta}_{r'},\kappa(\rho_{\zeta},7,0)\big)$ lies in $gH^h\setminus gH$. Therefore, for all sufficiently large $n$, we get a point other than $x_n$ on $\gamma_n$ in $gH^h\setminus gH$, contradicting the fact that $\{x_n\}=\gamma_n\cap gH^h$.
\end{proof}

\begin{theorem}\label{VDthm} Let $G$ be a finitely generated group and $\mathcal H$ be a finite collection of finitely generated subgroups of $G$. Suppose $H^h$ is contracting in $G^h$ for every $H\in\mathcal H$ and $\partial_cG^h$ is compact.
Then all combinatorial horoballs $gH^h$ have uniformly bounded visual size.

\end{theorem}
\begin{proof}
Let us first consider the combinatorial horoball $H^h$, where $H\in\mathcal H$.
We will prove that there exists $M\geq0$ such that  $diam(V_p(H^h))$ is at most $M$ for all $p\in G^h\setminus H^h$.\\
Suppose not, then there exists a sequence of points $\{p_n\}$ and points $x_n,x_n'\in V_{p_n}(H^h)$ with $diam(V_{p_n}(H^h))\geq d(x_n,x_n')\geq n$. Let $\gamma_n$ and $\gamma_n'$ be respective $(3,0)$-quasi-geodesics for points $x_n$ and $x_n'$ from $p_n$, as in the definition of $V_{p_n}(H^h)$. Note that $x_n,x'_n\in H$. Consider the points $q_n := x_n'^{-1}p_n$, $y_n:=x_n'^{-1}x_n$ and paths $\mu_n:=x_n'^{-1}\gamma_n$, $\nu_n:=x_n'^{-1}\gamma_n'$. Then $e$ (the identity element of $G$) and $y_n$ belongs to $V_{q_n}(H^h)$ and $\nu_n$, $\mu_n$ are respective $(3,0)$-quasi-geodesics for points $e,y_n$. Also $d(e,y_n)\geq n$. The limit set  $\Lambda(H^h)$ is a singleton set in $G^h$ and let $\Lambda(H^h)=\{\zeta\}$. Let $\alpha^{\zeta}$ be the vertical ray in $H^h$ starting from $e$. Then, after passing to a subsequence if necessary,  $y_n$ converges to $\zeta$. \par
First, consider the case where the sequence $\{d(e,q_n)\}$ is bounded.  Let $\beta_n:=[e,q_n]*\mu_n$. Since $l([e,q_n])$ is uniformly bounded, there exists $L\geq1$ and $A\geq0$ such that $\beta_n$ is a $(L,A)$-quasi-geodesic. By definition of convergence, given $r\geq1$ there exists number $N=N(r)$ such that for all $n\geq N$, $\beta_n$ comes $\kappa(\rho_{\zeta},L,A)$ close to $\alpha^{\zeta}[r,\infty)$. As $\{d(e,q_n)\}$ is bounded, for a sufficiently large $r$, a point of $\mu_n$ lies $\kappa(\rho_{\zeta},L,A)$ close to $\alpha^{\zeta}[r,\infty)$ and that point also lies in $H^h\setminus H$. This is a contradiction as we have taken $\mu_n\setminus \{y_n\}$ outside $H^h$. Therefore, we are left with the case that the sequence $\{d(e,q_n)\}$ is unbounded. By Proposition \ref{compactf}, there exists a subsequence of $\{q_n\}$ converging to a point $\eta\in\partial_cG^h$. Without loss of generality, we can assume $q_n\rightarrow \eta$.\par
\textbf{Case 1:} $\eta\neq \zeta$.\par
Let $z_n$ be a nearest point projection of $e$ onto $\mu_n$. Consider the path $\beta_n':=[e,z_n]*\mu_n|_{[z_n,y_n]}$. Each $\beta_n'$ is a $(7,0)$-quasi-geodesic. By Theorem \ref{Mtheorem}, $l([e,z_n])$ is uniformly bounded. Since $y_n\rightarrow \zeta$, by arguments same as in previous paragraph we get: for all large $n$, $\mu_n$ lies in $H^h\setminus H$; a contradiction.\par
\textbf{Case 2:} $\eta=\zeta$.\par
Consider paths $\nu_n$ connecting $e$ and $q_n$. Note that $\nu_n$ is a $(3,0)$-quasi-geodesic. Since $q_n\rightarrow \zeta$, by arguments similar to previous two cases we find: for all large $n$, $\nu_n$ has a point in $H^h\setminus H$, a contradiction.
\par
As the collection $\mathcal H$ of subgroups is finite and
the left action of any element of $G$ on $G^h$ is isometry, so we have the required result.
\end{proof}

\section{Main Result}\label{MainThm}
Let $X$ be a geodesic metric space. A geodesic triangle in $X$ is said to have a $\delta$-barycenter if there exists point $q\in X$ such that each side has distance at most $\delta$ from $q$. Point $q$ will be called a \textit{$\delta$-barycenter} for the triangle. $X$ is said to be hyperbolic if there exists a $\delta\geq0$ such that every geodesic triangle has a $\delta$-barycenter.
We refer the reader to see  Proposition 1.17 of Chapter III.H in \cite{bridson} and   Section 6 in \cite{Bowditchnotes} for different equivalent notions of hyperbolicity.
\begin{theorem}\label{Mthm2}
Let $G$ be a finitely generated group and  $\mathcal{H}=\{H_i\}$ be a finite collection of finitely generated infinite index subgroups of $G$. Then, $G$ is hyperbolic relative to $\mathcal H$ if and only if
every combinatorial horoball $H_i^h$ is contracting in the cusped space $G^h$ and $G^h$ has compact contracting boundary.
\end{theorem}

\begin{proof}
Suppose $G$ is hyperbolic relative to $\mathcal{H}$. Then the cusped space $G^h$
is a hyperbolic metric space, combinatorial horoballs are  quasiconvex in $G^h$
and hence $gH^h$ is contracting in $G^h$ for every $g\in G, H\in\mathcal{H}$. The combinatorial horoball $gH^h$ contracting and hyperbolic implies that any vertical geodesic ray in $gH^h$ is also contracting.\\
Conversely, to prove that $G$  is hyperbolic relative to $\mathcal{H}$, we need only to show that $G^h$ is a hyperbolic metric space.
Suppose $G^h$ is not a hyperbolic metric space, then there exists a sequence of positive reals $(\delta_n)$ and a
sequence of geodesic triangles $\{\triangle(x_n,y_n,z_n)\}$ in $G^h$ such that the triangle $\triangle(x_n,y_n,z_n)$ has a $\delta_n$-barycenter in $G^h$ with $\delta_n\rightarrow \infty$. Here $\delta_n$ is taken to be minimal i.e. if $\delta'_n<\delta_n$ then the $\triangle(x_n,y_n,z_n)$ has no $\delta'_n$-barycenter. Let $q_n\in G^h$ be a $\delta_n$-barycenter of $\triangle(x_n,y_n,z_n)$. Let $\gamma^1_n$ be a geodesic joining $x_n$ to $y_n$ , $\gamma^2_n$ be a geodesic joining $y_n$ to $z_n$ and $\gamma^3_n$ be a geodesic joining $z_n$ to $x_n$.\\
\textbf{Case 1:} $q_n$'s lie in $G$. \par
\noindent After multiplication with $q_n^{-1}$ barycenters of $\triangle(x_n,y_n,z_n)$ will come at $e$, the identity of $G$. So, without loss of generality, we can assume $e$ to be a $\delta_n$-barycenter of $\triangle(x_n,y_n,z_n)$. Let $p^1_n$, $p^2_n$ and $p^3_n$ be respective nearest point projection from $e$ onto sides $\gamma^1_n$, $\gamma^2_n$ and $\gamma^3_n$. Let $\alpha^{1}_{x_n}:=[e,p^1_n]*{\gamma^1_n}_{[p^1_n,x_n]}$, $\alpha^{1}_{y_n}:=[e,p^{1}_n]*{\gamma^{1}_n}_{[p^1_n,y_n]}$, similarly we define $\alpha^{2}_{y_n}$, $\alpha^{2}_{z_n}$, $\alpha^{3}_{z_n}$ and $\alpha^{3}_{x_n}$. Here “$*$” means concatenation of two paths. All these paths are continuous $(3,0)$-quasi-geodesic. Since $\delta_n\rightarrow \infty$, at least one of the sequence among $\{d(e,x_n)\}$, $\{d(e,y_n)\}$ and $\{d(e,z_n)\}$ is unbounded. If exactly one sequence is unbounded and rest are bounded then sequences $\{d(e,\gamma^1_n)\}$, $\{d(e,\gamma^2_n)\}$ and $\{d(e,\gamma^3_n)\}$ are bounded, contradiction to $\delta_n\rightarrow \infty$. Therefore we have two subcases:\par
\noindent\textbf{Subcase 1:}  The sequences $\{d(e,y_n)\}$ and $\{d(e,z_n)\}$ are unbounded while $\{d(e,x_n)\}$ is bounded. As $\bar{G^h}$ is sequentially compact, sequences $\{y_n\}$ and $\{z_n\}$ have subsequences converging to points in $\partial_{c}G^h$. Assume, without loss generality that $y_n\rightarrow \zeta_1$ and $z_n\rightarrow \zeta_2$ in $\partial_cG^h$. If $\zeta_1\neq\zeta_2$ then application of Theorem \ref{Mtheorem} gives that $\{d(e,\gamma^2_n)\}$ is bounded. Since $\{d(e,x_n)\}$ is bounded, we get that $\{d(e,\gamma^1_n)\}$, $\{d(e,\gamma^2_n)\}$, $\{d(e,\gamma^3_n)\}$ are bounded, which gives contradiction to $\delta_n\rightarrow \infty$.\par
Suppose $\zeta_1=\zeta_2=\zeta$. Given $r\geq1$, by definition of $x_n,y_n \rightarrow \zeta$, there exists $N=N(r)$ such that for $n\geq N$ \;
\begin{equation}
d\big(\alpha^{1}_{y_n},N^c_ro\cap \alpha^{\zeta}\big),d\big(\alpha^{2}_{y_n},N^c_ro\cap \alpha^{\zeta}\big),d\big(\alpha^{3}_{z_n},N^c_ro\cap \alpha^{\zeta}\big)\leq \kappa(\rho_{\zeta},3,0)
\end{equation}
There exists $r_i\geq r$ and $s_i>0$, where $i=1,2,3$, such that all the distances $d(\alpha^{\zeta}(r_1),\alpha^{1}_{y_n}(s_1)),
  d(\alpha^{\zeta}(r_2),\alpha^{2}_{y_n}(s_2)), d(\alpha^{\zeta}(r_3),\alpha^{3}_{z_n}(s_3))$ are at most $ \kappa(\rho_{\zeta},3,0)$.
  Let $r_m=\min\{r_1,r_2,r_3\}$. As $\alpha^{\zeta}$ is Morse, therefore there exists a number $\kappa_1=\kappa_1(\rho_{\zeta})>0$ and $t_i\leq s_i$
  such that the distances $d(\alpha^{\zeta}(r_m),\alpha^{1}_{y_n}(t_1)) $,  $d(\alpha^{\zeta}(r_m),\alpha^{2}_{y_n}(t_2))$,
   $ d(\alpha^{\zeta}(r_m),\alpha^{3}_{z_n}(t_3))$ are at most $\kappa_1$.
  The lengths of $\alpha^{1}_{y_n}|_{[0,t_1]}$, $\alpha^{2}_{y_n}|_{[0,t_2]}$, $\alpha^{3}_{z_n}|_{[0,t_3]}$ tend to infinity
  as $r\to\infty$.
  By choosing $r$ large enough we have $d\big(\alpha^{\zeta}(r_m),\gamma^i_n\big)< \delta_n$ for all $n\geq N(r)$ and for all $i\in\{1,2,3\}$. This contradicts the minimality of $\delta_n$ for $n\geq N(r)$. Therefore $\zeta_1\neq\zeta_2$.\\
\noindent\textbf{Subcase 2:} Suppose all the three Sequences $\{d(e,y_n)\}$, $\{d(e,z_n)\}$, $\{d(e,x_n)\}$ are unbounded. As in previous case, assume that $x_n\rightarrow \zeta_1$, $y_n\rightarrow \zeta_2$ and $z_n\rightarrow \zeta_3$ in $\partial_cG$. If $\zeta_1\neq \zeta_2\neq\zeta_3\neq \zeta_1$ then it will contradict, as in Case 1, to $\delta_n\rightarrow \infty$.
Suppose $\zeta_1=\zeta_2=\zeta$. Given $r\geq 1$, by definition of $x_n,y_n\rightarrow \zeta$, there exists $N=N(r)$ such that for all $n\geq N$
\begin{equation}
d\big(\alpha^{1}_{x_n},N^c_ro\cap \alpha^{\zeta}\big),d\big(\alpha^{2}_{y_n},N^c_ro\cap \alpha^{\zeta}\big),d\big(\alpha^{3}_{y_n},N^c_ro\cap \alpha^{\zeta}\big)\leq \kappa(\rho_{\zeta},3,0)
\end{equation}
From the above equation and arguments similar to \textbf{Subcase 1}, we get a large $r$ and a point on $\alpha^{\zeta}$ such that for all $n\geq N(r)$, the distance of that point from all three sides of $\triangle_n$ is strictly less than $\delta_n$. This contradicts the minimality of $\delta_n$ for $n\geq N(r)$, hence $\zeta_1\neq\zeta_2$. Similarly one can show that $\zeta_2\neq\zeta_3$ and $\zeta_3\neq \zeta_1$.\par
\noindent \textbf{Case 2:} $q_{n}$'s lie in horoballs. \\
The collection $\mathcal H$ of subgroups finite implies that there exists a subsequence $q_{n_k}$ of $\{q_n\}$ such that $q_{n_k}$ lies in translates of $H^h$ for some $H\in\mathcal H$.
After multiplication with suitable group elements, we can assume that  $q_{n_k}$'s lie in $H^h$. So, without loss of generality, we can assume that $q_n\in H^h$. Since $q_n\in H^h$, $q_n=(h_n,n')$ for some $h_n\in H$ and a non-negative integer $n'$. It is important to note that the sequence $\{n'\}$ of natural numbers is unbounded otherwise using the idea of case 1, one can  find a contradiction to $\delta_n\rightarrow \infty$. Therefore we assume $n'\rightarrow \infty$. First, consider the case when all sides of $\triangle_n$ penetrates $H^h$. Since $H^h$ is hyperbolic and contracting in $G^h$, by Theorem \ref{VDthm}, there exists $C_1\geq0$ such that $\triangle_n$ has a $C_1$-barycenter in $H^h$ for all $n$, contradiction to $\delta_n\rightarrow \infty$.\par
 Let $p_n^1$, $p_n^2$ and $p_n^3$ be a nearest point projection of $q_n$ onto $\gamma_n^1$, $\gamma_n^2$ and $\gamma_n^3$ respectively. Consider $(3,0)$-quasi-geodesics $\alpha^{1}_{x_n}:=[q_n,p^1_n]*{\gamma^1_n}_{[p^1_n,x_n]}$, $\alpha^{1}_{y_n}:=[q_n,p^{1}_n]*{\gamma^{1}_n}_{[p^1_n,y_n]}$, similarly define $\alpha^{2}_{y_n}$, $\alpha^{2}_{z_n}$, $\alpha^{3}_{z_n}$ and $\alpha^{3}_{x_n}$. Parameterize these continuous $(3,0)$-quasi-geodesic by it's arc length, with initial point $q_n$. Also take arc length parameterization of each $\gamma_n^i$ with $\gamma_n^1(0)=x_n$, $\gamma_n^2(0)=y_n$ and $\gamma_n^3(0)=z_n$. \par

Consider the case when exactly one side, say $\gamma_n^1$, doesn't penetrate $H^h$. Let $m_{[q_n,p^1_n]}$ and $m_{\gamma_n^3}$ be the respective last point of $\alpha^{1}_{x_n}$, and $\gamma_n^3$ in $H^h$. Also let $m_{\gamma_n^2}$ be the first point of $\gamma_n^2$ intersecting $H^h$. By Theorem \ref{VDthm}, there exists $C_2\geq0$ such that $d(m_{[q_n,p^1_n]}, m_{\gamma_n^3})$, $d(m_{[q_n,p^1_n]},\ m_{\gamma_n^2})\leq C_2$. Since $\{l_{[q_n,p^1_n]}\}$ is unbounded, for all large $n$, distance of $m_{[q_n,p^1_n]}$ to sides of $\triangle_n$ is $< \delta_n$, contradiction to the minimality of $\delta_n$.\par

Assume, exactly two sides say $\gamma_n^1$ and $\gamma_n^2$ does not penetrate. If $m_{[q_n,p^1_n]}$ and $m_{[q_n,p^2_n]}$ be respective last point of $\alpha^{1}_{x_n}$ and $\alpha^{2}_{y_n}$ in $H^h$. By Theorem \ref{VDthm}, there exists some $C_3\geq0$ such that $d(m_{[q_n,p^2_n]},m_{[q_n,p^1_n]})$, $d(m_{[q_n,p^2_n]},m_{\gamma_n^3})\leq C_3$. By arguments in previous, we get contradiction to minimality of $\delta_n$.\par
Lastly, consider the case when none of the sides penetrates $H^h$. Let $m_{[q_n,p^1_n]}$,$m_{[q_n,p^2_n]}$ and $m_{[q_n,p^3_n]}$ be the respective last point of $\alpha^{1}_{x_n}$, $\alpha^{2}_{y_n}$ and  $\alpha^{3}_{x_n}$ in $H^h$.  By Theorem \ref{VDthm}, diam$\{m_{[q_n,p^1_n]},m_{[q_n,p^2_n]},m_{[q_n,p^3_n]}\}$ is uniformly bounded. Therefore, as previously, this contradicts the minimality of $\delta_n$ for all large $n$.

\end{proof}

\section{Dynamics of relatively contracting element on the boundary}

\begin{definition} Let $G$ be a finitely generated and $G^h$ is a cusped space for some finite collection of finitely generated group. Consider the usual $G$ action on $G^h$. An element $g\in G$ is said to be relatively contracting  if the map $\mathbb{Z}\rightarrow G^h$ $:n\rightarrow g^n x_0$ is a quasi-isometric embedding and the image is a contracting subset of $G^h$ for some(hence any) $x_0\in G^h$.
\end{definition}
A relatively contracting element $g$ fixes two points of the boundary $\partial_cG^h$ which we denote by $g^{-\infty}$ and $g^{\infty}$.

\begin{remark}  If an element $g\in G$ is relatively contracting for $G$-action on $G^h$, then it is also contracting for $G$-action on itself. This follows from Theorem 1.1 of \cite{Tarik}.\end{remark}
A weaker version of North-South dynamics of  contracting elements of $G$ on $\partial_cG$ has been proved by Cashen-Mackay (See Lemmas 9.2, 9.3 and Theorem 9.4 of  \cite{cashen2017}) and  their  proof goes through verbatim in our case to prove the same weak North-South dynamics of relatively contracting elements.

\begin{theorem}(Weak north-south dynamics for relatively contracting elements)(Theorem 9.4 of \cite{cashen2017})\label{NSDynamics} Let $g\in G$ be a relatively contracting element. For every open set $V$ containing $g^{\infty}$ and every compact set $C\subseteq \partial_cG^h\setminus \{g^{-\infty}\}$ there exists an $N$ such that for all $n\ge N$ we have $g^nC\subseteq V$.
\end{theorem}

Next, we give an application of Weak North-South dynamics of relatively contracting element of $G$ on $\partial_cG^h$.

\begin{theorem}\label{dense}
Let $G$ be a finitely generated group and $\mathcal H$ be a finite collection of finitely generated subgroups of $G$ such that $H^h$ is contracting for all $H\in\mathcal H$. Let $\eta$ be a non-parabolic point i.e. $\eta$ is not the limit point of $gH_i^h$ for any $H_i\in\mathcal{H}$ and $g\in G$. Then the orbit $G\eta$ of $\eta$ is dense in $\partial_cG^h$.
\end{theorem}
\begin{proof}
\underline{Case 1:} Let $\zeta:=\Lambda_cgH_i^h$ for some $H_i\in\mathcal{H}$ and $g\in G$.   We will prove  that there exists a sequence $\{h_n\}$ in $gH_i$  $h_n\eta\rightarrow \zeta$.\par
Let $\beta$ be a bi-infinite geodesic joining $\zeta$ and $\eta$. Also, assume $\beta_0=e$, $\beta(-\infty)=\zeta$, $\beta(\infty)=\eta$ and $\beta(0,\infty)$ lie outside $gH_i^h$.
Let $\{h_n\}$ be a sequence in $gH_i$ such that $h_n\to\zeta$. Let $\rho$ be a sublinear function such that every subsegment of $\beta$ is $\rho$-contracting. Then $h_n\beta$ is also $\rho$-contracting. Let $\mu_n$ be geodesic starting from $e$ and representing $h_n\eta$. By Lemma \ref{almosttriangle}, there exists a sublinear function $\rho'$ independent of $n$ such that $\mu_n$ is $\rho'$-contracting. Let $p_n$ be a nearest point projection from $e$ onto $h_n\beta$. Define $\delta_n:=[e,p_n]\cup[p_n,h_n]$ where $[p_n,h_n]$ is the subsegment of $h_n\beta$ between $p_n$ and $h_n$.\par
 Next, we show that sequence $\{d(e,p_n)\}$ is unbounded. To see this first observe if $p_n$ lie on $gH_i^h$ then $\{d(e,p_n)\}$ is unbounded. So, assume $p_n$ lie outside the horoball. Clearly the subsegment $[p_n,h_n)$ of $h_n\beta$ lie outside $gH_i^h$. On contrary, assume $\{d(e,p_n)\}$ is bounded. $h_n\rightarrow \zeta$ and paths $\delta_n$ in this case is $(3,0)$-quasi-geodesic connecting $e$ and $h_n$. By definition of convergence, if sequence $\{d(e,p_n)\}$ is bounded, then for all large $n$ subsegment $[p_n,h_n)$ lie in $gH_i^h$, a contradiction.\par
 Given $r\geq1$, it suffices to consider $L$, $A$ such that $L\leq\sqrt{r/3}$, $A\leq r/3$. Let $\gamma_n$ be arbitrary continuous $(L,A)$-quasi-geodesic in $h_n\eta$. Since $h_n\rightarrow \zeta$ and sequence $\{d(e,p_n)\}$ is unbounded, for given $R\geq r$ there exists $N=N(R)$ such that for all $n\geq N$, $\alpha^{\zeta}_r$ is $\kappa'(\rho_{\zeta},3,0)$-close to $[e,p_n]$. Path $\nu_n:=[e,p_n]\cup[p_n,h_n\eta)$ is a $(3,0)$-quasi-geodesic asymptotic to $\gamma_n$. As geodesics in $h_n\eta$ is uniformly contracting and $L\leq\sqrt{r/3}$, $A\leq r/3$ there exists a constant $M$ such that Hausdorff distance between $\gamma_n$ and $\nu_n$ is atmost $M$. Therefore for all large $n$, $d(\alpha^{\zeta}_r,\gamma_n)\leq \kappa'(\rho_{\zeta},3,0)+M$. Finally, application of Lemma \ref{CM 4.6} gives that, for all large $n$, $\gamma_n$ is $\kappa(\rho_{\zeta},3,0)$-close to $\alpha^{\zeta}[r,\infty)$. Hence $h_n\eta\rightarrow \zeta$.\\
 \underline{Case 2 }: Let $\zeta$ be a non-parabolic point.  Let $\alpha^{\zeta}$ be a geodesic ray starting from $e$ and representing $\zeta$. Let $\alpha:$$=$$\alpha^{\zeta}$, then $\alpha$ does not eventually go into any horoball. Assume $\alpha$ is parameterized by it's arc length with $\alpha_0$$=$$e$, then there exists an increasing sequence $\{\sigma_n\}$ of natural numbers such that $\alpha_{\sigma_n}\in G$. Also, for some $g'\in G$, $g'\eta\neq\eta$. Let $g'\eta=\eta'$. Join $\eta$ to $\eta'$ by a geodesic $\beta$. $\beta$ passes through a group element.  Assume, without loss of generality $\beta_0$$=$$e$. Choose $\rho$ such that $\alpha$, $\beta_{[0,\infty)}$ and $\bar{\beta}_{[0,\infty)}$ are $\rho$-contracting.
For all large $n$, at most one of $\alpha_{\sigma_n}\beta_{[0,\infty)}$ and $\alpha_{\sigma_n}\bar{\beta}_{[0,\infty)}$ lies in closed $\kappa'(\rho,1,0)$ neighborhood of $\alpha_{[0,\sigma_n]}$ for distance greater than $2\kappa'(\rho,1,0)$ from $\alpha_{\sigma_n}$ (a proof is given in Claim 1 of Lemma \ref{metrizable}).
For each $n$, define $g_n$$:=$$\alpha_{\sigma_n}$ if $\alpha_{\sigma_n}\beta_{[0,\infty)}$ does not remain in $\kappa'(\rho,1,0)$ neighborhood of $\alpha_{[0,\sigma_n]}$ for distance greater than $2\kappa'(\rho,1,0)$ otherwise define $g_n$$:=$$\alpha_{\sigma_n}g'$. We repeat the same arguments as in the  proof of Claim 1 of Lemma \ref{metrizable} to prove that $g_n\eta$ converges to $\zeta.$
\end{proof}
\begin{lemma}\label{rational dense}Let $G$ be  finitely generated group
and $\mathcal H$ be a finite collection of finitely generated subgroups of $G$ such that $\abs{\partial_cG^h}\geq 2$. Assume $G$ contains a relatively contracting element. Then given a non-empty open set $U\in \partial_cG^h$ there exists a relatively contracting element $g\in G$ such that $g^{\infty}\in U$.
\end{lemma}
\begin{proof}
Let $a$ be a relatively contracting element of $G$. Then $a^{\infty}$ is a non-parabolic point.
By Theorem \ref{dense}, the orbit $G a^{\infty}$
of $a^{\infty}$ is dense in $\partial_cG^h$. There exists $w\in G$ such that $wa^{\infty}\in U$. Take $g=waw^{-1}$. Then $g$
is relatively contracting in $G$ and $g^{\infty}=wa^{\infty}\in U$.
\end{proof}
Taking $\mathcal H$ to be the collection of trivial subgroup in Lemma \ref{rational dense} yields the following corollary.
\begin{corollary}\label{rational dense cor}(Corollary 6.2,
\cite{QLiu})\label{Liu} Let $G$ be a non virtually cyclic finitely generated group such that $\partial_cG\neq \emptyset$. Assume $G$ contains a contracting element. Then following holds; given a non-empty open set $U\in \partial_cG$ there exists a contracting element $g\in G$ such that $g^{\infty}\in U$.
\end{corollary}
\section{Applications}
\begin{definition}(Cannon-Thurston map) Let $X, Y$ be proper geodesic spaces and $i:Y\hookrightarrow X$
be an embedding. We say that a  Cannon-Thurston map for the pair $(X,Y)$ exists if the embedding $i$ extends to a continuous map $\partial_ci:\partial_cY\to \partial_cX$.
\end{definition}

\begin{theorem} \label{CT} Let $G$ be a finitely generated group and $H\leq G$ be a normal hyperbolic subgroup such that Cannon-Thurston map for the pair $(G,H)$ exists. If $G$ contains a contracting element and $\Lambda_cH$ has at least two elements then $\Lambda_cH=\partial_cG$ and $G$ is a hyperbolic group.
\end{theorem}
\begin{proof}

As $H$ is a hyperbolic group, $\partial_cH$ is a compact set. Since Cannon-Thurston map for $(G,H)$ exists, $\Lambda_cH$ being continuous image of a compact set is compact.
Suppose $\Lambda_cH$ is strictly contained in $\partial_cG$. Let $U=\partial_cG\setminus \Lambda_cH$. Then $U$
is an open subset of $\partial_cG$.
By Corollary \ref{rational dense cor}, $U$ contains an
element $g^{\infty}$ for some contracting element $g\in G$.
As $H$
is normal in $G$, the limit set $\Lambda_c(H)$ of $H$ in $G$ is $G$-invariant. Let $\eta\in\Lambda_c(H)\setminus \{g^{-\infty}\}$. Then $g^n\eta$ lies in $\Lambda_c(H)$.
But by weak North-South dynamics (Theorem \ref{NSDynamics}),
$g^n\eta$ converges to $g^{\infty}$. Hence, $g^n\eta\in U=\partial_cG\setminus \Lambda_c(H)$ for all large $n$, this is a contradiction. Hence $\Lambda_cH=\partial_cG$. Thus, $\partial_cG$ is compact and hence $G$ is a hyperbolic group.
\end{proof}

Suppose we have a short exact sequence of pairs of finitely generated groups
\[
1\rightarrow (K,K_1)\stackrel{i}\rightarrow (G,N_G(K_1))\stackrel{p}{\rightarrow}(Q,Q_1)\rightarrow 1
\]
with  $K$  hyperbolic relative to $K_1$. Assume $G$ \textit {preserves cusps} of $K$ i.e. for all $g\in G$, there exists $k_g\in K$ such that $gK_1g^{-1}=k_gK_1k_g^{-1}$.
The embedding $i$ induces an embedding $i_h:K^h\hookrightarrow G^h$ between cusped spaces.
Combining Theorem 3.11 of \cite{pal} and  Proposition 5.7 of \cite{Mj-Sardar} by Mj. \& Sardar, implies that if $G$ is   hyperbolic relative to $N_G(K_1)$  then there exists a Cannon-Thurston map $\partial_c i_h:\partial_cK^h\to \partial_cG^h$ for the embedding $i_h$.

\begin{theorem}

Suppose we have a short exact sequence of pairs of finitely generated groups
\[
1\rightarrow (K,K_1)\stackrel{i}\rightarrow (G,N_G(K_1))\stackrel{p}{\rightarrow}(Q,Q_1)\rightarrow 1
\]
 with $K$  hyperbolic relative to $K_1$, $G$ preserves cusps of $K$ and the limit set $\Lambda_c(K^h)$ of $K^h$ in $\partial_cG^h$ consists of at least two points. If there exists a Cannon-Thurston map $\partial i_h$ for the embedding
$i_h: K^h\hookrightarrow G^h$ then $G$ is hyperbolic relative $N_G(K_1)$.
\end{theorem}
\begin{proof}
Let $\xi=\Lambda_c(K_1^h)\subset \partial_cK^h$ and $x_n\in K_1$ be such that $x_n\to \xi$ in $K^h$.
The sequence $i_h(x_n)$ converges to $\partial i_h(\xi)$  in $G\cup\partial_cG^h$ because $\partial i_h$ is a continuous extension of $i_h$. Thus, $\partial i_h(\xi)\in \Lambda_c(K^h)\subset \partial_cG^h$.
Let $\alpha$ be a geodesic ray in $G$ starting from the identity element $e$ representing $\partial i_h(\xi)$. Then $\alpha$ is contracting in $G^h$ and lies in $N_G(K_1)^h$. By Lemma \ref{horoball contracting}, $N_G(K_1)^h$ is contracting in $G^h$. As $K$ is hyperbolic relative to $K_1$, $\partial_cK^h$ is compact.
Since $\partial i_h$ is continuous and
$\partial i_h(\partial_cK^h)=\Lambda_c(K^h)$, the limit set $\Lambda_c(K^h)$  is compact.\\
Claim : $\Lambda_c(K^h)=\partial_cG^h$.\\
Suppose $\Lambda_c(K^h)$ is strictly contained in $\partial_cG^h$. Let $U=\partial_cG^h\setminus \Lambda_c(K^h)$. Then $U$ is an open subset of $\partial_cG^h$. By Lemma \ref{rational dense}, $U$ contains an
element $g^{\infty}$ for some relatively contracting element $g\in G$.
As $K$
is normal in $G$, the limit set $\Lambda_c(K^h)$ is $G$-invariant. Let $\eta\in\Lambda_c(K^h)\setminus \{g^{-\infty}\}$. Then $g^n\eta$ lies in $\Lambda_c(K^h)$.
But by weak North-South dynamics (Theorem \ref{NSDynamics}),
$g^n\eta$ converges to $g^{\infty}$. Hence, $g^n\eta\in U=\partial_cG^h\setminus \Lambda_c(K^h)$ for all large $n$, this is a contradiction. Therefore, $\Lambda_c(K^h)=\partial_cG^h$ and
hence $\partial_cG^h$ is compact. Therefore by Theorem \ref{Mthm2}, $G$
is hyperbolic relative to $N_G(K_1)$.

\end{proof}

\section*{\sc References}
 \addcontentsline{toc}{section}{Bibliography}
\small
\begingroup
\renewcommand{\section}[2]{}
% this must be set to use natbib (citep, citet) but requires BibTeX
\bibliographystyle{plainnat}

% number 99 determines how much citation can be included in file (maximum 99)

\end{document}